%% file: sn-article.tex
\documentclass{article}

\input{packages_commands}

\numberwithin{equation}{section}
\newtheorem{theorem}{Theorem}[section]
\newtheorem{prop}[theorem]{Proposition}
\newtheorem{defn}[theorem]{Definition}
\newtheorem{obs}[theorem]{Observation}
\newtheorem{algo}[theorem]{Algorithm}

\begin{document}

\title{\TheTitle}

\author{
    Caio O. da Silva \and
    Yuri A. Aoto \and
    Felipe F. G. S. Costa \and
    M{\'a}rcio F. da Silva
}

\date{}

\maketitle

\input{abstract.tex}

\input{intro.tex}

\input{rnm.tex}

\input{newton_st}

\input{newton_gr}

\input{newton_lm}

\input{hessian}

\input{remarks}

\input{hartree_fock}

\input{results}

\input{conclusions}

\begin{acknowledgements}

The authors thank the Dean's Office for Research of UFABC for providing the research facilities, and to the Coordena\c{c}{\~a}o de Aperfei\c{c}oamento de Pessoal de N{\'i}vel Superior (CAPES), Finance Code 001.
This research has been supported by grants \#2017/21199-0, \#2018/04617-6,
and \#2020/04891-0 from the S{\~a}o Paulo Research Foundation (FAPESP).

\end{acknowledgements}

\bibliographystyle{spmpsci}
\bibliography{sn-bibliography}

\end{document}

%% file: packages_commands.tex
\usepackage[a4paper, total={6in, 10in}]{geometry}
\usepackage[T1]{fontenc}
\usepackage[utf8]{inputenc}
\usepackage{amsfonts}
\usepackage{amsthm}
\usepackage{mathtools}
\usepackage{mathrsfs}
\usepackage{physics}
\usepackage{commath}
\usepackage{tabularx}
\usepackage{hyperref}
\usepackage{booktabs}
\usepackage{graphicx}
\usepackage{cite}
\usepackage{ebgaramond}

\newcommand*{\TheTitle}{A thorough study of Riemannian Newton's Method}

\DeclareMathOperator{\Gr}{Gr}
\DeclareMathOperator{\ST}{St}
\DeclareMathOperator{\Euc}{Euc}
\DeclareMathOperator{\Hel}{\widehat{H}_{el}}
\DeclareMathOperator{\oneh}{\widehat{h}}
\DeclareMathOperator{\riemhess}{Hess^{X}}
\DeclareMathOperator{\hessSt}{Hess^{\ST}}
\DeclareMathOperator{\hessGr}{Hess^{\Gr}}
\DeclareMathOperator{\riemgrad}{grad^{X}}
\DeclareMathOperator{\gradSt}{grad^{\ST}}
\DeclareMathOperator{\gradGr}{grad^{\Gr}}

\renewcommand{\epsilon}{\varepsilon}
\newcommand*{\eucgrad}[1]{\ensuremath{\mathbf{G}_{#1}}}
\newcommand*{\euchess}[1]{\ensuremath{\mathbf{H}_{#1}}}

\newcommand*{\jac}[1]{\ensuremath{\mathbf{J}_{#1}}}

\newcommand*{\id}{\mathrm{Id}}
\newcommand*{\trnp}[1]{\ensuremath{{#1}^\top}}

\newcommand*{\sym}[1]{\mathrm{sym}\!\del{#1}}
\newcommand*{\asym}[1]{\mathrm{skew}\!\del{#1}}

\newcommand{\ol}{\overline}
\newcommand{\Cmin}{\ensuremath{C_{\mathrm{min}}}}
\newcommand{\perm}[1]{\mathrm{perm}\!\del{#1}}
\newcommand{\vect}[1]{\mathrm{vec}\!\del{#1}}
\newcommand{\unvec}[3]{\mathrm{unvec}\!\del{#1, \, #2 \times #3}}
\newcommand{\mat}[2]{\mathrm{Mat}(#1 \times #2)}
\newcommand{\riemproj}{\mathrm{proj}}
\newcommand{\bn}{\mathbb{N}}
\newcommand{\br}{\mathbb{R}}
\newcommand{\ro}{\mathrm{O}}
\newcommand{\inner}[2]{\left\langle #1, \, #2 \right\rangle}
\newcommand{\x}{\times}
\newcommand{\mc}{\mathcal}
\newcommand{\mf}{\mathfrak}
\newcommand{\ceq}{\coloneqq}
\newcommand{\eps}{\varepsilon}
\newcommand{\covd}[1]{\frac{\mathrm{D}}{\dd{t}}{#1}}
\newcommand{\gr}[2]{\mathrm{Gr}(#1, #2)}
\newcommand{\st}[2]{\mathrm{St}(#1, #2)}

\makeatletter
\def\input@path{{sections/}}
\makeatother

\graphicspath{
  {images/}
}

\def\keywordname{{\bfseries Keywords}}
\def\keywords#1{\par\addvspace\medskipamount{\rightskip=0pt plus1cm
\def\and{\ifhmode\unskip\nobreak\fi\ $\cdot$
}\noindent\keywordname\enspace\ignorespaces#1\par}}

\def\ackname{Acknowledgements}
\def\acknowledgement{\par\addvspace{17pt}\small\rmfamily
\trivlist\if!\ackname!\item[]\else
\item[\hskip\labelsep
{\bfseries\ackname}]\fi}

\newenvironment{acknowledgements}{\begin{acknowledgement}}
{\end{acknowledgement}}

%% file: abstract.tex
\begin{abstract}
  This work presents a thorough numerical study of Riemannian Newton's Method (RNM) for optimization problems, with a focus on the Grassmannian and on the Stiefel manifold.
  We compare the Riemannian formulation of Newton's Method with its classical Euclidean counterpart based on Lagrange multipliers by applying both approaches to the important and challenging Hartree--Fock energy minimization problem from Quantum Chemistry.
  Experiments on a dataset of 125 molecules show that the Riemannian approaches achieve higher convergence rates, require fewer iterations, and exhibit greater robustness to the choice of initial guess.
  In this work we also analyze the numerical issues that arise from using Newton's Method on the total manifold when the cost function is defined on the quotient manifold.
  We investigate the performance of a modified RNM in which we ignore the small eigenvalues of the Hessian and the results indicate that this modified method is stable and performs on par with the RNM on the quotient manifold.
  
  \keywords{Newton's Method \and Lagrange multipliers \and Grassmannian \and Stiefel manifold \and Hartree--Fock \and Riemannian Optimization}
\end{abstract}

%% file: intro.tex
\section{Introduction} \label{sec:intro}

The field of Riemannian Optimization\footnote{Also known as \emph{Optimization on manifolds}.} is primarily interested in studying problems of the form:
\begin{equation}
    \text{minimize} \; f:X \to \br,
\end{equation}
being $(X, \inner{\cdot}{\cdot}_X)$ a Riemannian manifold and $f$ the cost function to be minimized.
By employing Riemannian counterparts of the gradient and Hessian of $f$, one obtains optimization algorithms that respect the underlying geometry of the manifold (\cite{boumal2023}).
This is relevant since many optimization problems can be formulated as:
\begin{equation} \label{eq:problem_eucl}
    \text{minimize} \; f:\br^n \to \br \;
    \text{subject to} \; c(x) = 0,
\end{equation}
being $c:\br^n \to \br^k$ a function whose inverse image $c^{-1}(0)$ is a submanifold of $\br^n$.
For this class of problems, Riemannian Optimization can be applied by considering the restriction of $f$ to $c^{-1}(0)$ directly.

In contrast, the traditional approach to tackle constrained optimization, known as the method of Lagrange multipliers, defines a Lagrangian function
\begin{align}
    \begin{split}
        \mathcal{L}:\br^n \times \br^k & \to \br \\
        (x, \epsilon) & \mapsto f(x) - \epsilon^{\top}c(x),
    \end{split}
\end{align}
and exploits the fact that solutions to~\eqref{eq:problem_eucl} correspond to saddle points of $\mc{L}$.
This raises a natural question:
which method performs better in practice, Riemannian Optimization or Lagrange multipliers?

In this work we perform a detailed comparison between the Riemannian version of Newton's Method and the classical Newton's Method combined with Lagrange multipliers.
This comparison is relevant because, although it is known that Riemannian methods can outperform Euclidean approaches that ignore the geometry of the constraints (\cite{boumal2011,vandereycken2013,xinru2020}), existing comparisons are usually limited to simpler optimization problems than the one considered in this work.
Moreover, to the best of our knowledge such a comparison has never been carried for Newton's Method.
That said, one of the goals of this paper is to bridge this gap by showing that the Riemannian Newton's Method outperforms its Lagrangian counterpart on a very challenging and relevant problem from Quantum Chemistry known as Hartree--Fock.

Another problem we address in this work is the following:
let $(\ol{X}, \inner{\cdot}{\cdot}_{\ol{X}})$ be a Riemannian manifold and $\sim$ an equivalence relation in $\ol{X}$.
Under suitable conditions (see \cite{boumal2023}), the quotient $X \coloneqq \ol{X}/\!\sim$ is also a Riemannian manifold, known as the quotient manifold.
Its Riemannian metric, $\inner{\cdot}{\cdot}_X$, is induced from $\inner{\cdot}{\cdot}_{\ol{X}}$ by lifting tangent vectors from $T_{[p]}X$ to $T_p\ol{X}$ via the inverse of the canonical projection $\pi:\ol{X} \to X$.\footnote{$[p] \in X$ denotes the equivalence class of $p \in \ol{X}$.}
Now, given a cost function $f:X \to \br$, one can always consider its pullback $\ol{f}:\ol{X} \to \br$ defined by $\ol{f} = f \circ \pi$.
This raises another question:
what are the differences between optimizing $f$ directly on the quotient manifold, $X$, or instead $\ol{f}$ on the total space $\ol{X}$?

It has been established that optimization algorithms that rely solely on the gradient of $f$, such as gradient descent, are equivalent for both formulations, see \cite{boumal2023}.
However, Newton's Method on the total manifold $\ol{X}$ is not equivalent to its counterpart on the quotient $X$ (also discussed in \cite{boumal2023}), and a key reason is that the Hessian of $\ol{f}$ may be singular, potentially leading to numerical instabilities during optimization.
Although this is known from a theoretical perspective, concrete examples analyzing this issue in practice are missing.
We aim to fill this gap.
Towards that end, we adapt to the Riemannian setting a classical spectrum truncation technique for Hessians and show that it enables stable optimization of $\ol{f}$ instead of $f$.
This is relevant because, in practice, it is often easier to formulate the problem and to derive the necessary formulas for functions defined on the total space $\ol{X}$ than on the quotient manifold $X$.

The third contribution of this paper is to provide implementations of the Riemannian Hessian.
Although formulas for the Hessian can be found in textbooks \cite{boumal2023, absil2008}, and a detailed computation of the Hessian on the Stiefel manifold can be found in \cite{sato2017}, it is still not straightforward to translate their formulas into a working implementation of Newton's Method.
In this work we provide two implementations for the Hessian on the Grassmannian and the Stiefel manifold (because the Hartree--Fock problem is defined on them), but the reasoning behind these implementations can be applied to other manifolds and cost functions.

The fourth and final contribution of this work is a numerical investigation of the convergence neighborhoods of the methods under consideration.
To decouple algorithmic performance from the influence of the initial guess, we estimate the convergence neighborhood of each method numerically and compare the results.
Our findings show that, on average, Riemannian methods possess larger convergence neighborhoods and are less sensitive to the choice of initial guess, suggesting they are more robust than Newton's Method combined with Lagrange multipliers.

The remainder of the paper is structured as follows:
Section~\ref{sec:rnr} introduces the Riemannian version of Newton's Method, beginning with a brief review of essential concepts of Riemannian Geometry.
Subsections~\ref{subsec:newton_st},~\ref{subsec:newton_gr}, and~\ref{subsec:newton_lm} present the explicit formulas necessary to implement the method on the Stiefel manifold, the Grassmannian, and using Lagrange multipliers, respectively.
In Section~\ref{sec:hessian}, we discuss in details the implementation of the Riemannian Hessian and how it enables optimization in the total manifold.
Section~\ref{sec:remarks} contains some important remarks regarding the similarities and differences between the methods implemented in this paper.
Section~\ref{sec:hartree_fock} provides an overview of the Hartree--Fock problem.
Finally, Section~\ref{sec:results} presents several empirical results comparing the performance of the three methods in practice

%% file: rnm.tex
\section{Riemannian Newton's Method} \label{sec:rnr}

In this section we will define what the Riemannian Newton's Method (RNM) is.
We will assume the reader is familiar with Differential Geometry, but, since the Riemannian Optimization literature diverges a little from the traditional Differential Geometry literature in the definition of some objects, we will provide such definitions and some reminders in what follows.
For more details, see \cite{lee2018, boumal2023}.

From now on, $(X, \inner{\cdot}{\cdot}_X)$ is assumed to be a Riemannian manifold, with $\inner{\cdot}{\cdot}_X$ being the Riemannian metric, $T_xX$ is the tangent space of $X$ at $x \in X$, and
\begin{equation}
  TX \ceq \bigsqcup_{x \in X} \cbr{x} \x T_xX
\end{equation}
is the tangent bundle of $X$.
For convenience, we will often omit $X$ from the metric $\inner{\cdot}{\cdot}_X$ and we will refer only to $X$ as the Riemannian manifold.

\begin{defn}
  Let $f:X \to \br$ be a smooth function defined on a Riemannian manifold $X$.
  The \emph{Riemannian gradient} of $f$ is the vector field $\riemgrad{f}:X \to TX$ uniquely defined by the following identity:
  \begin{equation}
    Df(x)(v) = \inner{v}{\riemgrad{f}(x)}
  \end{equation}
  for all $(x, v) \in TX$.
  $Df(x)(v)$ denotes the \emph{directional derivative} of $f$ at $x$ in the direction $v$.
\end{defn}

Recall that a Riemannian metric $\inner{\cdot}{\cdot}_X$ always induces a connection in $X$ known as \emph{Levi-Civita connection}.
We will denote such connection by $\nabla^X$.
A connection allows us to differentiate vector fields and if we apply this to the Riemannian gradient, we obtain the concept of a Riemannian Hessian.

\begin{defn}
  Let $f:X \to \br$ be a smooth function defined on a Riemannian manifold $(X, \inner{\cdot}{\cdot}_X)$ and let $\nabla^X$ be the associated Levi-Civita connection.
  The \emph{Riemannian Hessian} of $f$ at $x \in X$ is the operator $\riemhess{f}(x):T_xX \to T_xX$ defined as:
  \begin{equation}
    \riemhess{f}(x)(v) = \nabla_v^X \riemgrad{f}.
  \end{equation}
\end{defn}

It is worth noting that the Riemannian gradient and Hessian coincide with the usual Euclidean gradient and Hessian one learns in Calculus when $X = \br^n$.
However, this is not true if $X \neq \br^n$ because the geometry of $X$ should be taken into account.
Often $X$ is a set of matrices satisfying some constraints and in these cases the Riemannian gradient and Hessian are computed by projecting the Euclidean gradient and Hessian onto the tangent space of $X$.
For more details, see \cite{absil2008, boumal2023}.

\begin{obs}
From now on $I \subset \br$ will be used to denote an interval.
The interval can be open or closed depending on the context.
Notice that in both cases it can be seen as a manifold, but, if it is a closed interval, then it is a manifold with boundary.
In both cases the concepts below are well-defined.
\end{obs}

So far we have two important concepts in optimization:
the gradient of a function $f$, which is a vector pointing in the direction that maximizes $f$, and the Hessian of $f$, which allows us to determine whether a critical point of $f$ is a saddle point, a maximum or a minimum.
However, we still need to know how we can move from one point to another because the Newton's Method is iterative.
To do this, we need geodesics,\footnote{Or, more generally, retractions (see \cite{boumal2023,absil2008}), but we will focus on geodesics in this work.}
which can be seen as a generalization of straight lines on manifolds.
Thus, we need to describe a straight line in terms of derivatives because that is what Differential Geometry is about.
Based on the fact that straight lines in the Euclidean space are curves whose second derivative is zero, we will see how to translate this to the Riemannian setting.

\begin{defn}
  Given a Riemannian manifold $X$ and a smooth curve $\gamma:I \to X$, we say that a function $Z:I \to TX$ is a \emph{smooth vector field along $\gamma$} if $Z$ is a smooth function and $Z(t) \in T_{\gamma(t)}X$ for all $t \in I$.
  We will denote the set of all smooth vector fields along $\gamma$ by $\mf{X}(\gamma)$.
\end{defn}

Notice that $\gamma'$ is a smooth vector field along $\gamma$.
However, the different concepts of derivatives we have so far are still not enough to differentiate $\gamma'$.
To do that, we need covariant derivatives.

\begin{defn}
  Let $X$ be a Riemannian manifold with Levi-Civita connection $\nabla^X$ and $\gamma:I \to X$ a smooth curve.
  The \emph{covariant derivative} induced by $\gamma$ is the unique operator $\covd:\mf{X}(\gamma) \to \mf{X}(\gamma)$ that satisfies the following properties for all $Y, Z \in \mf{X}(\gamma)$, $U \in \mf{X}(X)$, $a, b \in \br$ and $g \in C^{\infty}(I)$:
  \begin{enumerate}
    \item $\covd{(aY + bZ)} = a\covd{Y} + b\covd{Z}$;
    \item $\covd{(gZ)} = g'Z + g\covd{Z}$;
    \item $\del{\covd{(U \circ \gamma)}}(t) = \nabla_{\gamma'(t)}^XU$.
  \end{enumerate}
\end{defn}

Now we can finally define geodesics.
\begin{defn}
  A smooth curve on a Riemannian manifold $\gamma:I \to X$ is called a \emph{geodesic} if, and only if,
  \begin{equation}
    \gamma''(t) \ceq \frac{\mathrm{D}}{\dd{t}}\gamma'(t) = 0
  \end{equation}
  for all $t \in I$.
\end{defn}

The next proposition assures us that geodesics always exist (locally) no matter in which direction we want to go.
\begin{prop}
  Given a Riemannian manifold $X$, $x \in X$ and $v \in T_xX$, there exists a unique maximal geodesic $\gamma_v:I \to X$ such that $0 \in I$, $\gamma_v(0) = x$ and $\gamma_v'(0) = v$.
  Maximal in this context means that if there exists an interval $J$ with the same properties, then $J \subset I$.
\end{prop}

\begin{proof}
  See \cite[Corollary 4.28]{lee2018}.
\end{proof}

To conclude, we need to define the Riemannian exponential, which allows us to move from one point to another on a manifold.
\begin{defn}
  Let $X$ be a Riemannian manifold and
  \begin{equation}
    \mc{U} \ceq \cbr{(x, v) \in TX : \text{$\gamma_v$ is defined on an interval containing $[0, 1]$}}.
  \end{equation}
  The \emph{Riemannian exponential} is the function $\exp^X:\mc{U} \to X$ defined by:
  \begin{equation}
    \exp^X(x, v) \ceq \gamma_v(1).
  \end{equation}
  We will usually denote $\exp^X(x, v)$ as $\exp_x^X(v)$.
\end{defn}

Putting all these concepts together, we can define the main algorithm we are interested in the present work.
\begin{algo}[RNM] \label{algo:riem_newton}
  Let $X$ be a Riemannian manifold and $f:X \to \br$ a cost function.
  Given a starting point $x_0 \in X$, proceed as follows:
  \begin{enumerate}

  \item Compute $\riemgrad{f}(x_k)$ and $\riemhess{f}(x_k)$;

  \item Solve
    $\riemhess{f}(x_k)(v_k) = -\riemgrad{f}(x_k)$
    for $v_k \in T_{x_k}X$;

  \item Update the point according to
    $x_{k+1} \ceq \exp_{x_k}^X(v_k)$.

  \end{enumerate}
  At the end of this procedure, one obtains a sequence $\cbr{x_k \in X : k \in \bn}$ that, under certain hypothesis (see \cite{smith2014}), converges to a critical point of $f$.
\end{algo}


%% file: newton_st.tex
\subsection{Stiefel manifold} \label{subsec:newton_st}

Given a symmetric positive-definite matrix $S \in \mat{d}{d}$ and $N \in \mathbb{N}$ satisfying $0 \le N \le d$, the Stiefel manifold is defined as follows:
\begin{equation}
  \st{d}{N} \ceq \cbr{C \in \mat{d}{N} : C^{\top}SC = \id_N}.
\end{equation}
It is worth pointing out that $S$ is usually assumed to be $\id_d$, but, due to the application described in Section~\ref{sec:hartree_fock}, we will consider this generalized version.
To implement the RNM in the Stiefel manifold we need to know the explicit formulas for some objects.
Starting with the tangent space to $\st{d}{N}$ at $C \in \st{d}{N}$, we have:
\begin{equation} \label{eq:tg_sp_st}
  T_C\st{d}{N} = \cbr{\mu \in \mat{d}{N} : C^{\top}S\mu + \mu^{\top}SC = 0_N}.
\end{equation}
The Riemannian metric we should consider in this generalized Stiefel manifold is:
\begin{equation} \label{eq:metric_st}
  \inner{\mu}{\mu'} = \tr(\mu^{\top}S\mu'),
\end{equation}
being $\mu', \mu \in \mat{d}{N}$.
Now, defining
\begin{equation}
  \sym{M} \ceq \frac{M + M^{\top}}{2}, \quad \asym{M} \ceq \frac{M - M^{\top}}{2},
\end{equation}
we have that the projection to the tangent space to the Stiefel manifold $\riemproj_C^{\ST}:\mat{d}{N} \to T_C\st{d}{N}$ is given by:
\begin{equation} \label{eq:proj_st}
  \riemproj_C^{\ST}(M) = M - C\sym{C^{\top}SM} = (\id_d - CC^{\top}S)M + C\asym{C^{\top}SM}.
\end{equation}
The Riemannian gradient of a smooth cost function $f:\st{d}{N} \to \br$ is given by:
\begin{equation} \label{eq:grad_st}
  \gradSt{f}(C) = \riemproj_C^{\ST}\del{\eucgrad{\ol{f}}(C)},
\end{equation}
being $\ol{f}$ an extension of $f$ to $\mat{d}{N}$ and $\eucgrad{\ol{f}}(C)$ the usual Euclidean gradient of $\ol{f}$.
The (Riemannian) Hessian of $f$ is given by:
\begin{equation} \label{eq:hessian_st}
  \hessSt{f}(C)(\mu)
  = \riemproj_C^{\ST}\del{S^{-1}\euchess{\ol{f}}(C)(\mu)
  - \mu\,\sym{\trnp{C}\eucgrad{\ol{f}}(C)}},
\end{equation}
being $\euchess{\ol{f}}(C)$ the Euclidean Hessian of $\ol{f}$.
We would like to make some comments regarding this formula for the Hessian.
First, since $\euchess{\ol{f}}(C) \in \mat{dN}{dN}$, $S^{-1} \in \mat{d}{d}$ and $\mu \in \mat{d}{N}$, we cannot interpret $S^{-1}\euchess{\ol{f}}(C)(\mu)$ as a regular matrix multiplication because there is a mismatch between the dimensions.
To fix this we need to (1) vectorize $\mu$;
(2) calculate $\euchess{\ol{f}}(C)(\mu)$ as a regular matrix-vector multiplication;
(3) unvectorize the resulting vector, obtaining a matrix in $\mat{d}{N}$;
(4) left-multiply by $S^{-1}$.
Another problem is that to solve Newton's Equation~\eqref{eq:newton} using a package such as \verb+NumPy+, we need to move $\mu$ to the right of the Hessian and this is quite tricky to do.
The reader can take a look at Section~\ref{sec:hessian} for more details on how to fix both of these problems.

To conclude, the Riemannian exponential is:
\begin{equation} \label{eq:exp_st}
  \exp_C^{\ST}(\mu)
  =
  \begin{bmatrix}
    C & & \mu
  \end{bmatrix}
  \exp\del{
    \begin{bmatrix}
      C^{\top}S\mu & & -\mu^{\top}S\mu \\
      \id_N & & C^{\top}S\mu
    \end{bmatrix}
  }
  \begin{bmatrix}
    \id_N \\ 0_N
  \end{bmatrix}
  \exp(-C^{\top}S\mu).
\end{equation}
For more details, see \cite{boumal2023, dasilva2024}.


%% file: newton_gr.tex
\subsection{Grassmannian} \label{subsec:newton_gr}

Let
\begin{equation}
  \ro(N) \ceq \cbr{M \in \mat{N}{N} : M^{\top}M = \id_N}
\end{equation}
and $\sim$ be the following equivalence relation on $\st{d}{N}$: given $C, C' \in \st{d}{N}$, we say that $C \sim C'$ if, and only if, there exists $M \in \ro(N)$ such that $C' = CM$.
The Grassmannian is defined as:
\begin{equation}
  \gr{N}{d} \ceq \st{d}{N}/\!\!\sim\; = \st{d}{N}/\ro(N).
\end{equation}

In order to implement the RNM in the Grassmannian, we once again need to know the explicit formulas for some objects.
Starting with the tangent space to $\gr{N}{d}$ at a point $C \in \gr{N}{d}$, we have:\footnote{We will denote a point in the Grassmannian by $C$ instead of the common notation $[C]$.}
\begin{equation}\label{eq:tg_sp_gr}
  T_C\gr{N}{d} = \cbr{\eta \in \mat{d}{N} : C^{\top}S\eta = 0},
\end{equation}
being $S \in \mat{d}{d}$ the same symmetric positive-definite matrix used in the definition of the Stiefel manifold.
The Riemannian metric is the same as in the Stiefel manifold.
Now, the projection to the tangent space to the Grassmannian $\riemproj_C^{\Gr}:\mat{d}{N} \to T_C\gr{N}{d}$ is given by:
\begin{equation}\label{eq:proj_gr}
  \riemproj_C^{\Gr}(M) = (\id_d - CC^{\top}S)M.
\end{equation}
Notice that $T_C\gr{N}{d}$ is a subspace of $T_C\st{d}{N}$ and the projections onto the Stiefel manifold and onto the Grassmannian are related by:
\begin{equation}
  \riemproj_C^{\ST}(M) = \riemproj_C^{\Gr}(M) + C\asym{C^{\top}SM}.
\end{equation}

Before we move on to the gradient and the Hessian, we need to define a lift.
Given a function $f:\gr{N}{d} \to \br$, the lift of $f$ is the function $\ol{f}:\st{d}{N} \to \br$ defined by $\ol{f} = f \circ \pi$, being $\pi:\st{d}{N} \to \gr{N}{d}$ the canonical projection.
Having defined that, the Riemannian gradient of a smooth cost function $f:\gr{N}{d} \to \br$ is given by:
\begin{equation}\label{eq:grad_gr}
  \gradGr{f}(C) = \riemproj_C^{\Gr}\del{\eucgrad{\ol{f}}(C)},
\end{equation}
being $\ol{f}$ an extension to $\mat{d}{N}$ of a lift of $f$ to $\st{d}{N}$, and $\eucgrad{\ol{f}}(C)$ the usual Euclidean gradient of $\ol{f}$.\footnote{The lift of the extension should be denoted by $\ol{\ol{f}}$, but to avoid cluttered notation we will write $\ol{f}$.}
The (Riemannian) Hessian of $f$ is given by:
\begin{equation}\label{eq:hessian_gr}
  \hessGr{f}(C)(\eta)
  = \riemproj_C^{\Gr}\del{\euchess{\ol{f}}(C)(\eta)} - \eta\,C^{\top}\eucgrad{\ol{f}}(C),
\end{equation}
being $\euchess{\ol{f}}(C)$ the Euclidean Hessian of $\ol{f}$.
Observe that, if we want to implement this equation, we will face problems similar to those already described in the previous subsection.
The details on how to implement the Hessian are in Section~\ref{sec:hessian}.

To conclude, we need the Riemannian exponential.
Let $O \in \mat{d}{d}$ be an invertible matrix such that $O^{\top}SO = \id_d$.
Then, given $\eta \in T_C\gr{N}{d}$ and letting $O^{-1}\eta = UDV$ be a \emph{thin Singular Value Decomposition},\footnote{Assuming $N \leq d$ (which is always the case for us), we have $U \in \mat{d}{N}$ and $D, V \in \mat{N}{N}$ such that $D$ is diagonal and $U^{\top}U = V^{\top}V = VV^{\top} = \id_N$. The matrix $U$ is also known as semi-orthogonal because it is false that $UU^{\top} = \id_d$ if $N < d$.} we have:
\begin{equation}\label{eq:geodesic_Gr}
  \exp_C^{\Gr}(\eta) = \qty(C \trnp{V} \cos(D) + O U \sin(D)) V.
\end{equation}
For more details, see \cite{boumal2023, dasilva2024}.


%% file: newton_lm.tex
\subsection{Newton's Method with Lagrange multipliers} \label{subsec:newton_lm}

The traditional approach to tackle constrained optimization problems is using Lagrange multipliers combined with the Euclidean Newton's Method.
However, this is just a particular case of the Riemannian version.
To see that, suppose we have the following optimization problem:
\begin{equation}
  \text{minimize} \; \ol{f}:\br^n \to \br \;
  \text{subject to} \; c(x) = 0,
\end{equation}
being $0 \in \br^k$ a regular value of the function $c:\br^n \to \br^k$ that encodes the constraints.
According to the Lagrange Multipliers Theorem (\cite[Theorem 12.1]{nocedal2006}), we can equivalently look for saddle points of the following Lagrangian:
\begin{equation}
  \begin{split}
    \mathcal{L}:\br^n \x \br^k & \to \br \\
    (x, \eps) & \mapsto \ol{f}(x) - \eps^{\top}c(x).
  \end{split}
\end{equation}
So, considering the manifold $X = \br^n \x \br^k$ and the cost function $\mc{L}$, we can see that, if we use Newton's Method to find the saddle points of $\mc{L}$, then the Lagrange multipliers approach is just a particular case of the Riemannian Newton's Method.\footnote{It is possible to use other algorithms to look for saddle points, but the most common is Newton's Method.}
To distinguish from the Riemannian version, though, we will call this approach Newton's Method with Lagrange multipliers (NMLM).

Moving to our particular case, the manifold we are interested in is $X = \mat{d}{N} \x \mat{N}{N}$ and the cost function is the following Lagrangian:
\begin{equation}
  \mc{L}(C, \, \eps) = \ol{f}(C) - \tr(\trnp{\eps}\qty(\trnp{C}SC - \id_N)).
\end{equation}
Notice that the matrix $\trnp{C}SC - \id_N$ is symmetric and, consequently, $\eps$ should also be.
This is useful to know because we need less parameters to optimize, 
but obtaining Equations~\eqref{eq:lagr_hc} and \eqref{eq:lagr_mixed} below
is more complicated if we ignore the entries above or below the diagonal of $\eps$.
It would also be possible to deduce the formulas assuming explicitly that $\eps$ is symmetric, but then the constraints itself would form a manifold.
We could use even more Lagrange multipliers to impose the symmetry of $\eps$, but we observed experimentally that, if we start with $\eps$ symmetric,
it remains symmetric in all iterations.
We could not prove this fact, though, and we opted not to impose any constraint on $\eps$ because the formulas are still easy to write down and because in practice $\eps$ remains symmetric anyway.

As already mentioned, in this case the Riemannian and Euclidean gradient and Hessian coincide, and they are given by:
\begin{equation} \label{eq:Lagr_grad}
  \eucgrad{\mc{L}}(C, \, \eps) =
  \begin{bmatrix}
    \eucgrad{\ol{f}}(C) - 2SC\eps \\
    -(\trnp{C}SC - \id_N)
  \end{bmatrix}
\end{equation}
and
\begin{equation} \label{eq:Lagr_hess}
  \euchess{\mc{L}}(C, \, \eps) =
  \begin{bmatrix}
    \euchess{\ol{f}}(C) - \euchess{c}(C) & -\jac{\trnp{C}SC}(C) \\
    -\trnp{(\jac{\trnp{C}SC}(C))} & 0_{N^2} \\
  \end{bmatrix},
\end{equation}
being $\euchess{c}(C)$ the Hessian of the constraint $\tr(\trnp{\eps}\qty(\trnp{C}SC - \id_N))$ and $\jac{\trnp{C}SC}(C)$ the Jacobian of the function $C \mapsto \trnp{C}SC$.
Note that in the block $\euchess{\ol{f}}(C) - \euchess{c}(C)$ we are considering only the second-order partial derivatives of $\mc{L}$ with respect to $C$ while in the block $-\jac{\trnp{C}SC}$ we have the mixed second-order partial derivatives with respect to $C$ and $\eps$.
The block made of zeros is the second-order partial derivatives of $\mc{L}$ with respect to $\eps$.
Observe also that in Equation \eqref{eq:Lagr_grad} we have a matrix instead of a vector, which means that to solve Newton's Equation~\eqref{eq:newton} one has to vectorize the gradient before plugging it into a solver such as \verb|numpy.linalg.solve|.
Let us now see how to construct the Hessian.

Ignoring $\euchess{\ol{f}}(C)$ because it depends on the specific problem one is trying to solve, we have:
\begin{equation} \label{eq:lagr_hc}
  \euchess{c}(C) = (\eps + \trnp{\eps}) \otimes S,
\end{equation}
being $\otimes$ the Kronecker product, and:
\begin{equation} \label{eq:lagr_mixed}
  \jac{\trnp{C}SC}
  = \id_N \otimes \trnp{C}S
  + \begin{bmatrix}
    \trnp{C}S \otimes e_1
    & \dots &
    \trnp{C}S \otimes e_N
    \end{bmatrix}.
\end{equation}
Moreover, the exponential in this case is quite easy because it is just a sum of vectors:
\begin{equation}
  \exp_{(C, \, \eps)}^{\Euc}(v, w) = (C + v, \; \eps + w).
\end{equation}
For more details, see \cite{dasilva2024}.

To conclude, observe that, if $(C^*, \, \eps^*)$ is a saddle point of $\mc{L}$,
then $\eps^*$ can be obtained in terms of $C^*$.
Indeed, since $\eucgrad{\mc{L}}(C^*, \, \eps^*) = 0$, we know from Equation~\eqref{eq:Lagr_grad} that $\eucgrad{\ol{f}}(C^*) - 2SC^*\eps^* = 0$ and, consequently, $2(C^*)^\top SC^*\eps^* = (C^*)^\top \eucgrad{\ol{f}}(C^*)$.
However, using Equation~\eqref{eq:Lagr_grad} again we know that $(C^*)^\top SC^* = \id_N$, which gives us the following identity:
\begin{equation} \label{eq:lagmult_from_C}
    \eps^* = \frac{1}{2} (C^*)^\top \eucgrad{\ol{f}}(C^*).
\end{equation}

%% file: hessian.tex
\section{Implementing the Riemannian Hessian} \label{sec:hessian}

Recall that in Newton's Method we need to find $v \in T_xX$ such that
\begin{equation} \label{eq:newton}
  \riemhess{f}(x)(v) = -\riemgrad{f}(x).
\end{equation}
We will call this equation \emph{Newton's Equation} throughout this paper.
If we are planning to solve this equation using a package such as \verb+NumPy+, we need to address two issues:
(1) the Riemannian Hessian for the Stiefel manifold (Equation~\eqref{eq:hessian_st}) and for the Grassmannian (Equation~\eqref{eq:hessian_gr}) are not just multiplying the tangent vector on the left, which is the format assumed by most solvers;
(2) the solution $v$ obtained does not necessarily satisfy the constraint $v \in T_xX$.
In what follows, we will provide two detailed implementations of the Hessian in the Stiefel manifold that fix both issues.
However, as it will become clear, both approaches can be adapted to other manifolds as well.

\subsection{Intrinsic Hessian} \label{subsec:intrinsic_hessian}

Given an orthonormal basis $\cbr{v_1, \ldots, v_d}$ for a vector space $V$ endowed with inner product $\inner{\cdot}{\cdot}$, we can write any vector $v \in V$ as:
\begin{equation} \label{eq:v_as_lincomb_basis}
  v = \sum_{i=1}^d \inner{v}{v_i}v_i.
\end{equation}
We also know that any operator $T:V \to V$ can be represented as a matrix as follows:
\begin{equation} \label{eq:matrix_rep}
  [T]_{ij} = \inner{v_i}{T(v_j)},
\end{equation}
being $i, j = 1, \ldots, d$.
Consequently, if $(X, \inner{\cdot}{\cdot})$ is a Riemannian manifold and $\cbr{v_1, \ldots, v_d}$ is an orthonormal basis for $T_xX$, the matrix representation of $\riemhess{f}(x)$ is:
\begin{equation} \label{eq:H_ij_innerprod}
  [\riemhess{f}(x)]_{ij} = \inner{v_i}{\riemhess{f}(x)(v_j)},
\end{equation}
being $i, j = 1, \ldots, d$.

It is straightforward to implement Equation~\eqref{eq:H_ij_innerprod}, but it is not very efficient because we have to iterate over the basis and compute $\riemhess{f}(C)(v_j)$ for all $j=1,\ldots,d$, which can be costly.
The advantage of this implementation is that we do not need extra dimensions and terms as we do in the other one we will see in Subsection~\ref{subsec:extrinsic_hessian}.
However, this implementation is intrinsic to the manifold, which allows us to analyze the spectrum of the Hessian (see Section~\ref{sec:results}).
Moreover, it is easy to check that the implementation is correct because of the following identity:
\begin{equation} \label{eq:hess_deriv}
  [\riemhess{f}(x)]_{ii}
  = \inner{v_i}{\riemhess{f}(x)(v_i)}
  = \od[2]{}{t}f(\gamma_{v_i}(0)),
\end{equation}
being $\gamma_{v_i}(t)$ the geodesic of $X$ starting at $x$ in the direction $v_i$.
To verify this, we computed the right-hand side using finite differences with a step size of $10^{-4}$ and confirmed that the result satisfies Equation~\eqref{eq:hess_deriv} within a tolerance of $10^{-5}$.
Since both implementations of Newton's Method (the one presented in this subsection and the one in the next subsection) produce identical results (within the numerical precision of floating-point operations), it indicates that our implementations are indeed correct.

Another advantage of this implementation of the Hessian is that we can disregard small or negative eigenvalues.
Consider the case in which the (symmetric) $\riemhess{f}(x)$ operator has $d$ orthonormal eigenvectors $u_i$ with corresponding eigenvalues $\lambda_i$ such that $\lambda_i > \delta > 0$, being $\delta > 0$ a chosen cutoff threshold.
Then, the solution to Newton's Equation can be approximated as:
\begin{equation} \label{eq:solve_newton_eigenvalues}
  v = -\sum_{\lambda_j > \delta} \frac{\inner{\riemgrad{f}(x)}{u_j}}{\lambda_j} u_j.
\end{equation}
We will refer to this method as the \emph{modified Riemannian Newton's Method} (mRNM) and we will see in Section~\ref{sec:results} that the choice of $\delta$ has a significant impact in the performance of Newton's Method on the Stiefel manifold.

Before we move forward to explain how we compute the Hessian for the Stiefel manifold and the Grassmannian, let us adapt Equation~\eqref{eq:H_ij_innerprod} to the case in which the basis is indexed by two indices.
Let $\mc{B} \ceq \cbr{v_{ij} : i \in J, \, j \in J'}$ be a basis for a subspace of $\mat{d}{N}$.
Then, to compute the Hessian, we need to order this basis.
There are two main conventions for this:
running first over $i$ (rows) and then over $j$ (columns) or vice-versa.
We will adopt the convention of running first over the rows in this paper.
By doing this, we can compute the matrix representation of the Hessian as follows:
\begin{equation} \label{eq:H_ijkl_innerprod}
  [\riemhess{f}(x)]_{(i+j \cdot |J|, \, k+l \cdot |J|)} = \inner{v_{ij}}{\riemhess{f}(x)(v_{kl})},
\end{equation}
being $|J|$ the cardinality of the set $J$.
One important observation to the reader interested in implementing this is that $[\riemhess{f}(x)] \in \mat{|J||J'|}{|J||J'|}$, which means that the solution to Newton's Equation~\eqref{eq:newton} and the gradient should be written in the basis $\mc{B}$ and not in the canonical basis of $\mat{d}{N}$.

Moving to the Stiefel manifold and to the Grassmannian, we know that (see \cite{edelman1998}) any vector $\mu \in T_C\st{d}{N}$ can be written as:
\begin{equation}
  \mu = C\nu + \eta,
\end{equation}
being $\nu \in \mat{N}{N}$ a skew-symmetric matrix and $\eta \in T_C\gr{N}{d}$.
Moreover, any $\eta \in T_C\gr{N}{d}$ can be written as:
\begin{equation}
  \eta = C_{\perp}\xi,
\end{equation}
being $\xi \in \mat{d-N}{N}$ and $C_{\perp} \in \mat{d}{d-N}$ such that $\begin{bmatrix} C & C_{\perp} \end{bmatrix} \in \st{d}{d}$.
Now, if we consider the canonical basis of $\mat{d-N}{N}$ formed by the matrices $E_{kl}$ such that all entries are $0$ except $(k, l)$, which is $1$, we obtain an orthonormal basis for $T_C\gr{N}{d}$:
\begin{equation}
  \mc{B}^{\Gr} \ceq \cbr{C_{\perp}E_{kl} : k=1,\ldots,d-N, \, l=1,\ldots,N}.
\end{equation}
Ordering this basis as mentioned above, we can easily compute the Hessian intrinsic to the Grassmannian.
For the Stiefel manifold, on the other hand, we have that
\begin{equation}
  \mc{B}^{\ST} \ceq \cbr{C\asym{\sqrt{2}E_{ij}} : i=2,\ldots,N, \, j=1,\ldots,i-1} \cup \mc{B}^{\Gr}
\end{equation}
is an orthonormal basis for $T_C\st{d}{N}$.
Ordering this basis by running first over $C\asym{E_{ij}}$ and then over $\mc{B}^{\Gr}$, we can compute the Hessian intrinsic to the Stiefel manifold as well.

Before moving to the extrinsic implementation, let us make two remarks:
(a) in order to compute $C_{\perp}$, one can initialize this matrix randomly and then use the Gram--Schmidt algorithm to obtain its columns.
Observe that we need to initialize the Gram--Schmidt algorithm with the columns of $C$, though, because in the end we want to obtain $\begin{bmatrix} C & C_{\perp} \end{bmatrix} \in \st{d}{d}$.
Moreover, we also need to use the inner product defined in Equation~\eqref{eq:metric_st} to take the matrix $S$ into account.
(b) In practice, even though the solution to Newton's Equation~\eqref{eq:newton} is written in terms of an intrinsic basis, the elements of the basis are usually written in terms of the canonical basis of an underlying Euclidean space.
As an example, we have $N(d-N)$ elements in the basis of the Grassmannian, $\mc{B}^{\Gr}$, but each element in this basis, $C_{\perp}E_{kl}$, lives in $\mat{d}{N}$.
So, when we compute the tangent vector obtained from solving Newton's Equation, it will also live in $\mat{d}{N}$.
Indeed, if each $\eta_{kl} \in \br$ represents the coefficient associated to $C_{\perp}E_{kl}$, then
\begin{equation}
  \eta = \sum_{k=1}^{d-N}\sum_{l=1}^{N} \eta_{kl} C_{\perp}E_{kl} \in \mat{d}{N}.
\end{equation}

\subsection{Extrinsic Hessian} \label{subsec:extrinsic_hessian}

The extrinsic implementation is motivated by the case in which we do not have an explicit basis for the tangent space and by a faster and direct construction of the Hessian.
Observe that, even if we do not have an explicit basis for the tangent space of a manifold $X$, the manifold itself is usually a submanifold of $\br^d$ or a quotient of a submanifold (see \cite{boumal2023}).
Consequently, very often the tangent space can be described as $T_xX = \cbr{v \in \br^n : h(x)v = 0_k}$, being $h:\br^n \to \mat{k}{n}$ a function of $x$.
Moreover, it is usually not hard to obtain a formula for the Hessian of a function defined on $X$ that can be applied to any $v \in \br^n$ instead of only $v \in T_xX$ (see Equations~\eqref{eq:hessian_st} and~\eqref{eq:hessian_gr}, for example).
In this case we can see $T_xX$ as a subspace of $\br^n$ and use the canonical basis of $\br^n$ to construct the (extended version of the) Hessian.\footnote{This should not be confused with the Hessian of the extension $\ol{f}$.}
Explicitly:
\begin{equation} \label{eq:extrinsic_hessian}
  [\riemhess{f}(x)]_{ij} = \inner{e_i}{\riemhess{f}(x)(e_j)}.
\end{equation}
However, in this representation the solution to Newton's Equation~\eqref{eq:newton} may not be in $T_xX$.
To fix this issue, we can force the solution to be in $T_xX$ by augmenting the Hessian and the gradient as follows:
\begin{equation} \label{eq:augmented_newton}
  \begin{bmatrix}
    [\riemhess{f}(x)] \\
    h(x)
  \end{bmatrix}
  v
  =
  -\begin{bmatrix}
    [\riemgrad{f}(x)] \\
    0_k
  \end{bmatrix}.
\end{equation}
This augmented version of Newton's Equation is also straightforward to solve using a solver such as \verb|numpy.linalg.solve| once we know the matrix representation of the Hessian and who $h$ is.
Let us now show how this can be done for the Stiefel manifold.

The first observation we have to make is that the tangent space in this case is a subspace of $\mat{d}{N}$.
Consequently, we will adopt the same convention used in the previous subsection, i.e., the Hessian will be computed according to Equation~\eqref{eq:H_ijkl_innerprod}.
By doing this, observe that we need to vectorize Equation~\eqref{eq:augmented_newton}, which means we actually want to solve
\begin{equation} \label{eq:augmented_matrix_newton}
  \begin{bmatrix}
    [\hessSt{f}(C)] \\
    h(C)
  \end{bmatrix}
  \vect{\mu}
  =
  -\mathrm{vec}
  \begin{bmatrix}
    [\gradSt{f}(C)] \\
    0_N
  \end{bmatrix},
\end{equation}
being $\mathrm{vec}$ the \emph{vectorization} operation defined as:
\begin{equation}
  M \ceq
  \begin{bmatrix}
    M_{11} & \ldots & M_{1n} \\
    \vdots & \ddots & \vdots \\
    M_{m1} & \ldots & M_{mn}
  \end{bmatrix}
  \leadsto
  \vect{M} \ceq
  \begin{bmatrix}
    M_{11} \\
    \vdots \\
    M_{m1} \\
    M_{12} \\
    \vdots \\
    M_{mn}
  \end{bmatrix}.
\end{equation}
Now that we know which equation we have to solve and how to solve it, let us start figuring out who $h$ is so we can plug it in Equation~\eqref{eq:augmented_matrix_newton}.

Recalling that
\begin{equation}
  T_C\st{d}{N} = \cbr{\mu \in \mat{d}{N} : C^{\top}S\mu + \mu^{\top}SC = 0_N},
\end{equation}
we can vectorize the constraint $\mu$ to find $h$.
Using the very handy identity $\vect{ABC} = (C^{\top} \otimes A)\vect{B}$, we have:\footnote{Again, $\otimes$ denotes the Kronecker product.}
\begin{align}
  \begin{split}
    \vect{C^{\top}S\mu + \mu^{\top}SC}
    & = \vect{C^{\top}S\mu} + \vect{\mu^{\top}SC}\\
    & = \vect{C^{\top}S\mu\id_N} + \vect{\id_d\mu^{\top}SC}\\
    & = (\id_N \otimes C^{\top}S)\vect{\mu}
    + (C^{\top}S \otimes \id_d)\vect{\mu^{\top}}.
  \end{split}
\end{align}
The first term is already as we want, but the second we need to develop further.
Observe that when we multiply a matrix $M$ and a vector $v$ and we swap two rows of the vector, say $v_i$ and $v_j$, the result of this product is the same as if we had swapped the $i$-th and $j$-th columns of $M$ instead of the rows of $v$.
Consequently, since $\vect{\mu^{\top}}$ is just a permutation of the rows of $\vect{\mu}$, we can permute the columns of $(C^{\top}S \otimes \id_d)$ instead.
Denoting this operation as $\mathrm{perm}$, we have that:
\begin{equation}
  h(C) = \id_N \otimes C^{\top}S + \perm{C^{\top}S \otimes \id_d}.
\end{equation}
To implement the equation above, instead of accessing the columns of $C^{\top}S \otimes \id_d$ in the order $[0, 1, 2, \ldots]$, we access it according to $[i + j \cdot d]$, running first over $j=0,\ldots,N-1$ and then over $i=0,\ldots,d-1$.

Now let us move to the Hessian.
Recall that in the Stiefel manifold the Hessian is given by:
\begin{equation} \label{eq:hessian_st_recall}
  \hessSt{f}(C)(\mu)
  = \riemproj_C^{\ST}\del{S^{-1}\euchess{\ol{f}}(C)(\mu)
  - \mu\,\sym{\trnp{C}\eucgrad{\ol{f}}(C)}},
\end{equation}
being
\begin{equation}
  \riemproj_C^{\ST}(M)
  = M - C\sym{C^{\top}SM}
  = M - C\frac{C^{\top}SM + M^{\top}SC}{2}.
\end{equation}
Vectorizing the projection using the same ideas presented above, we obtain the following:
\begin{equation}
  \frac{1}{2}\del{\id_N \otimes (2\id_d - CC^{\top}S) - \perm{C^{\top}S \otimes C}}\vect{M}.
\end{equation}
Moving to what is inside parenthesis in Equation~\eqref{eq:hessian_st_recall}, we need to handle the mismatch in the dimensions of $S^{-1} \in \mat{d}{d}$, $\euchess{\ol{f}}(C) \in \mat{dN}{dN}$ and $\mu \in \mat{d}{N}$.
To fix this issue, we need to interpret $\euchess{\ol{f}}(C)(\mu)$ as the product $\euchess{\ol{f}}(C)\vect{\mu} \in \br^{dN}$ and then, in order to multiply by $S^{-1}$, we need to unvectorize this vector to $\unvec{\euchess{\ol{f}}(C)\vect{\mu}}{d}{N}$.
Now that we fixed this issue, we can vectorize everything inside parenthesis, obtaining:
\begin{equation}
  \del{(\id_N \otimes S^{-1})\euchess{\ol{f}}(C) - \sym{C^{\top}\eucgrad{\ol{f}}(C)} \otimes \id_d}\vect{\mu}.
\end{equation}
Putting everything together, the final matrix representation of the Hessian for the Stiefel manifold is:
\begin{multline}
  [\hessSt{f}(C)] = 
    \frac{1}{2}\del{\id_N \otimes (2\id_d - C\trnp{C}S)
    - \perm{\trnp{C}S \otimes C}} \\
    \cdot \del{(\id_N \otimes S^{-1}) \euchess{\ol{f}}(C)
    - \sym{\trnp{C}\eucgrad{\ol{f}}(C)} \otimes \id_d}.
\end{multline}
Moreover,
\begin{equation}
  [\hessSt{f}(C)]_{(i+j \cdot d, \, k+l \cdot d)}
  = \inner{E_{ij}}{\hessSt{f}(C)(E_{kl})},
\end{equation}
showing that the vectorization procedure described above is equivalent to computing the Hessian in the canonical basis of the underlying Euclidean space.
However, it is (arguably) easier to compute the full matrix representation directly than to compute entrywise.

The Hessian for the Grassmannian can be computed analogously.
In this case we have:
\begin{equation}
  h(C) = \del{\id_N \otimes C^{\top}S}
\end{equation}
and the Hessian is:
\begin{equation}
  [\hessGr{f}(C)]
  =
  \del{\id_N \otimes (\id_d - CC^{\top}S)S^{-1}}\euchess{\ol{f}}(C)
  - (\eucgrad{\ol{f}}(C))^{\top}C \otimes \id_d.
\end{equation}


%% file: remarks.tex
\section{Some important remarks} \label{sec:remarks}

The Grassmannian encodes two constraints:
$C^{\top}SC = \id_N$ and the equivalence relation given by $C \sim C'$ if, and only if, $C = C'M$ for some $M \in \ro(N)$.
The first constraint is easy to encode using Lagrange multipliers (see Subsection~\ref{subsec:newton_lm}).
However, to the best of our knowledge, it is unclear whether the equivalence relation can also be encoded using Lagrange multipliers.
Consequently, a more appropriate comparison to Newton's Method with Lagrange multipliers is the version of Newton's Method defined on the Stiefel manifold, rather than on the Grassmannian.

In practice the situation is more complicated because the cost function to be optimized (including our cost function) may be defined in the Grassmannian (or any quotient manifold) and lifting this function to the Stiefel manifold (or to the total space, in general) leads to numerical issues (see \cite[Chapter 9]{boumal2023}).
While it might seem reasonable to consider Lagrange multipliers as an alternative, as we will show in Section~\ref{sec:results}, this approach is outperformed by the Riemannian methods.

Having said that, it is not always evident that an optimization problem is defined on a quotient manifold (see \cite{bernard2021}).
Moreover, it is usually easier to compute the Riemannian gradient, Hessian, and exponential map on submanifolds of a Euclidean space than on quotient manifolds.
With this in mind, in Subsection~\ref{subsec:intrinsic_hessian} we showed how to adapt the well known technique of truncating the spectrum of the Hessian to the Riemannian setting and in Section~\ref{sec:results} we show that this technique mitigates the numerical issues associated with the Hessian and that it leads to good empirical performance.

From an implementation standpoint, both Riemannian methods are very similar.
In fact, as discussed in (\cite[Chapter 9]{boumal2023}), the gradient of a cost function $f:X \to \br$ and its lift $\ol{f}:\ol{X} \to \br$ are the same.
Moreover, the Hessians are also very similar.
In the total space it is given by:
\begin{equation}
  \mathrm{Hess}^{\ol{X}}{\ol{f}}(x)(v)
  =
  \riemproj_x^{\ol{X}}\del{
    \riemproj_x^{\ol{X}}\del{\euchess{\ol{\ol{f}}}(x)(v)}
    +
    \riemproj_x^{\ol{X}}(v)\eucgrad{\ol{\ol{f}}}(x)
  },
\end{equation}
while in the quotient it is given by:
\begin{equation}
  \mathrm{Hess}^X{f}(x)(v)
  =
  \riemproj_x^X\del{
    \riemproj_x^{\ol{X}}\del{\euchess{\ol{\ol{f}}}(x)(v)}
    +
    \riemproj_x^{\ol{X}}(v)\eucgrad{\ol{\ol{f}}}(x)
  }.
\end{equation}

As one can see, the only difference is the outer projection.
One may interpret the fact that the Hessians are not equivalent due to the fact that, in general, $\mathrm{grad}^{\ol{X}}\ol{f}(x) \neq \mathrm{grad}^{\ol{X}}\ol{f}(y)$ even when $x \sim y$.\footnote{Both gradients project to $\mathrm{grad}^Xf([x])$, though.}
Consequently, since the Hessian on the total space measures variations of the gradient in directions $v \in T_x\ol{X} = T_{[x]}X \oplus V_x$, it may not vanish when acting on vectors orthogonal to $T_{[x]}X$.

It is also worth mentioning that we cannot, in general, lift geodesics from a quotient manifold to the total space in such a way that, given $[x] \in X$, $v \in T_{[x]}X$ and a geodesic $\gamma_v:T \to X$ emanating from $([x], v)$, we have:
\begin{equation} \label{eq:equiv_exp}
  \mathrm{exp}^X([x], v) = [\mathrm{exp}^{\ol{X}}(x, v)].
\end{equation}
However, in practice the quotient structure of $X$ is usually obtained via a Lie group action (satisfying some requirements, see \cite[Theorem 9.18]{boumal2023}) and in this case the result holds (see \cite{delhoyo2016}).
Since this is the case for the Grassmannian, with the Lie group being $\ro(N)$, Equation~\eqref{eq:equiv_exp} holds.

Overall, despite the fact that the Riemannian Newton's Method on a quotient manifold and on its respective total space are not equivalent, in practice they can behave quite similarly, as shown in Section~\ref{sec:results} when the spectrum of the Hessian on the Stiefel manifold is truncated.


%% file: hartree_fock.tex
\section{Hartree--Fock} \label{sec:hartree_fock}

The application we are interested in this work is the \emph{Hartree--Fock Method}, which we will define as a constrained optimization problem that aims to solve an approximation of the Schrödinger equation.
To describe this problem mathematically, let a \emph{Slater determinant} be a wave function (i.e., an element of $L^2(\br^{3N})$, being $N \in \bn$ the number of electrons in the system) that can be written as:
\begin{equation}
  \begin{split}  
    (\phi_1 \wedge \ldots \wedge \phi_N)(\vb{r}_1, \ldots, \vb{r}_N)
    & \ceq \frac{1}{\sqrt{N!}}\det\begin{bmatrix} \phi_1(\vb{r}_1) & \ldots & \phi_N(\vb{r}_1) \\ \vdots & \ddots & \vdots \\ \phi_1(\vb{r}_N) & \ldots & \phi_N(\vb{r}_N) \end{bmatrix} \\
    & = \frac{1}{\sqrt{N!}}\sum_{\sigma \in S_N} \mathrm{sign}(\sigma)\phi_{\sigma(1)}(\vb{r}_1) \ldots \phi_{\sigma(N)}(\vb{r}_N),
  \end{split}
\end{equation}
being $\phi_i \in L^2(\br^3)$.
Denoting by $\nabla_{\vb{r}_i}^2$ the Laplacian that considers only the derivatives with respect to $\vb{r}_i = (r_{1i}, r_{2i}, r_{3i}) \in \br^3$, we define the \emph{electronic Hamiltonian} as:
\begin{equation}
  \Hel
  \ceq \sum_{i=1}^N \del{
    -\frac{1}{2}\nabla_{\vb{r}_i}^2
    - \sum_{j=1}^M \frac{Z_j}{\abs{\vb{r}_i - \vb{R}_j}}
    + \sum_{j > i}^N \frac{1}{\abs{\vb{r}_i - \vb{r}_j}}
  },
\end{equation}
being $M$ the number of nuclei in the system and $Z_j$ the atomic number of the $j$-th nucleus (after we fix an order for the nuclei).
It is important to mention that $\vb{R}_j \in \br^3$ is assumed to be fixed\footnote{This is known as the \emph{Born--Oppenheimer approximation}.} and that $\Hel$ is defined on (a subspace of) $L^2(\br^{3N})$.
Now, fixing a finite set $\mathscr{B} \ceq \qty{\psi_1, \ldots, \psi_d} \subset L^2(\br^3)$ such that
\begin{equation}
  S_{ij} \ceq \inner{\psi_i}{\psi_j} = \int\limits_{\br^3} \psi_i(\vb{r})\psi_j(\vb{r})\dd\vb{r},
\end{equation}
we define Hartree--Fock as the following constrained optimization problem:
\begin{alignat}{2}
    & \text{minimize}   & \quad & \inner{\phi}{\Hel(\phi)}, \\
    & \text{subject to} & & \phi = \phi_1 \wedge \ldots \wedge \phi_N, \\
    &                   & & \phi_j \ceq c_{1j}\psi_1 + \ldots + c_{dj}\psi_d, \\
    &                   & & \begin{bmatrix}
      c_{11} & \ldots & c_{1N} \\
      \vdots & \ddots & \vdots \\
      c_{d1} & \ldots & c_{dN}
    \end{bmatrix}  \in \gr{N}{d} \,.
\end{alignat}

Observe that, since $\mathscr{B}$ is fixed, this problem is indeed defined in the Grassmannian.
We will write down the cost function in terms of $c_{ij}$ in a moment, but before that let us explain the physical meaning behind this problem.
The cost function we are minimizing represents the electronic energy of a system with $2N$ electrons (a molecule, for example).
We are considering only systems with $2N$ electrons instead of $N$ because in this case we can treat both spins of an electron equally and, consequently, use the same Slater determinant for both spins.
The constraint serves two purposes:
it ensures that the wave function remains orthonormal and it encodes the fact that state spaces in Quantum Mechanics are defined as the projective space of $L^2(\br^{3N})$.
See \cite{dasilva2024} for the details of this method, including the unrestricted version in which the spins can differ.
Let us now describe this minimization problem explicitly so we can compute the derivatives of the cost function.

Let $\oneh(\vb{r})$ represent the \emph{one-electron operator}, which is defined as follows:
\begin{equation}
  \oneh(\vb{r})\psi(\vb{r})
  =
  - \frac{1}{2}\nabla^2\psi(\vb{r})
  - \sum_{j=1}^M \frac{Z_j}{\abs{\vb{r} - \vb{R}_j}}\psi(\vb{r}).
\end{equation}
The motivation behind this operator is that $\psi \in L^2(\br^3)$ represents an electron, the first term measures the kinetic energy of this electron and the second term measures the electric potential energy of the electron with respect to all nuclei of a given system.
Exploiting the basis set $\mathscr{B}$, we can define the following matrices:
\begin{equation}
  h_{ij} \ceq \int\limits_{\br^3} \psi_i(\vb{r})\oneh(\vb{r})\psi_j(\vb{r})\dd\vb{r},
\end{equation}
which represents the one-electron operator,
\begin{equation}
  g_{ijkl} \ceq \int\limits_{\br^6} \psi_i(\vb{r}_1)\psi_j(\vb{r}_2)
  \frac{1}{\abs{\vb{r}_1 - \vb{r}_2}} \psi_k(\vb{r}_1)\psi_l(\vb{r}_2)
  \dd\vb{r}_1\dd\vb{r}_2,
\end{equation}
which represents a two-electron operator that measures the Coulomb interaction between two electrons, and
\begin{equation}
  F_{ij} \ceq h_{ij} + \sum_{k,l=1}^d \sum_{\nu=1}^N c_{k\nu}c_{l\nu}\qty(2g_{ikjl} - g_{ijkl}),
\end{equation}
which is known as the \emph{Fock matrix}.
Using all these matrices, we can rewrite the Hartree--Fock problem as follows:
\begin{alignat}{2}
  & \text{minimize}   & \quad & f(C) \ceq \sum_{i,j=1}^d \sum_{\mu=1}^N c_{i\mu}c_{j\mu}(h_{ij} + F_{ij}), \\
  & \text{subject to} & & C \ceq \begin{bmatrix}
    c_{11} & \ldots & c_{1N} \\
    \vdots & \ddots & \vdots \\
    c_{d1} & \ldots & c_{dN}
  \end{bmatrix} \in \gr{N}{d}.
\end{alignat}
See \cite{szabo96} for more details about the method from a Chemistry perspective.

To conclude, we need to compute the Euclidean gradient and Hessian of the (lift/extension of the) function $f$, which we will denote as $\ol{f}$.\footnote{Observe that in this case we can use the exact same expression for $\ol{f}$ and $f$, but to keep the notation compatible with the previous sections we will use $\ol{f}$.}
The Euclidean gradient of $\ol{f}$ is given by:
\begin{equation}\label{eq:grad_hf}
  \eucgrad{\ol{f}}(C) = 2FC,
\end{equation}
and the entries of the Euclidean Hessian, $\euchess{\ol{f}}$, are given by:
\begin{equation} \label{eq:hessian_energy}
  \pdv{}{x_{ki}}{x_{lj}}\ol{f}\del{C}
  = 2 \qty(F_{kl} \delta_{ij} + \sum_{r,s=1}^d c_{ri} c_{sj}
  (2g_{krls} - g_{kslr} - g_{klrs})),
\end{equation}
being $\delta_{ij}$ the Kronecker delta function.
We could not find a direct computation for the matrix $\euchess{\ol{f}}$, which means we construct it iteratively using the equation above and the convention adopted in Equation~\eqref{eq:H_ijkl_innerprod} for the indices.
See \cite[Chapter 4]{dasilva2024} for more details.
That is it, we have everything we need to implement the algorithms described in the previous sections to solve the Hartree--Fock problem.


%% file: results.tex
\section{Results} \label{sec:results}

All methods discussed in this paper are available in the package called \emph{Grassmann} \cite{aotograssmann2023}.
They were used to solve the Hartree--Fock energy minimization problem for the subset of molecules in the G2/97 dataset \cite{curtiss1997} with an even number of electrons.
This subset, consisting of 125 molecules, will be referred simply as \emph{the dataset} throughout this section.

For convenience, we summarize below the notation used to refer to the methods compared in this section:

\begin{itemize}
  
  \item \textbf{NMLM}: Newton's Method with Lagrange Multipliers (Section~\ref{subsec:newton_lm});
  
  \item \textbf{RNM-Gr}: Riemannian Newton's Method on the Grassmannian (Section~\ref{subsec:newton_gr});
  
  \item \textbf{RNM-St}: Riemannian Newton's Method on the Stiefel manifold (Section~\ref{subsec:newton_st});
  
  \item \textbf{mRNM-St$_{\delta}$}: modified Riemannian Newton's Method on the Stiefel manifold, with a cutoff $\delta$ used to discard small eigenvalues (see Equation~\eqref{eq:solve_newton_eigenvalues}).
  When not specified, $\delta$ defaults to $10^{-8}$.

\end{itemize}

We verified numerically that solving Newton's Equation~\eqref{eq:newton} using either the intrinsic or extrinsic Riemannian Hessian (see Section~\ref{sec:hessian}) produces the exact same steps for both manifolds.

Except for Subsection~\ref{sec:conv_neigh_water}, all methods used the \emph{superposition of atomic densities} (SAD) procedure to obtain the initial guess.
This initialization is standard in Hartree--Fock calculations, see~\cite{lehtola2019} for details.
Moreover, the Lagrange multipliers were initialized as $\eps = C^\top F C$, following Equations~\eqref{eq:lagmult_from_C} and~\eqref{eq:grad_hf}.

\input{res_water}

\input{spectrum}

\input{res_statistics}

\input{ablation}

\input{profiling}

\input{res_conv_viz}

\input{profiling_r}

%% file: res_water.tex
\subsection{Water} \label{sec:res_water}

We begin by examining the convergence behavior of the different methods on the water molecule, as one can see in Figure~\ref{fig:h2o_conv}.
The top panel shows the iterations on the $x$-axis and the energy difference to the minimum on the $y$-axis.
Since there is no ground truth minimum, we use the solution obtained by RNM-Gr as a reference and convergence is defined by the gradient norm dropping below $10^{-8}$.
The bottom panel shows the gradient norm at each iteration and low values indicate how well the method converged.

\begin{figure}[h!]
    \centering
    \includegraphics[width=1.0\textwidth]{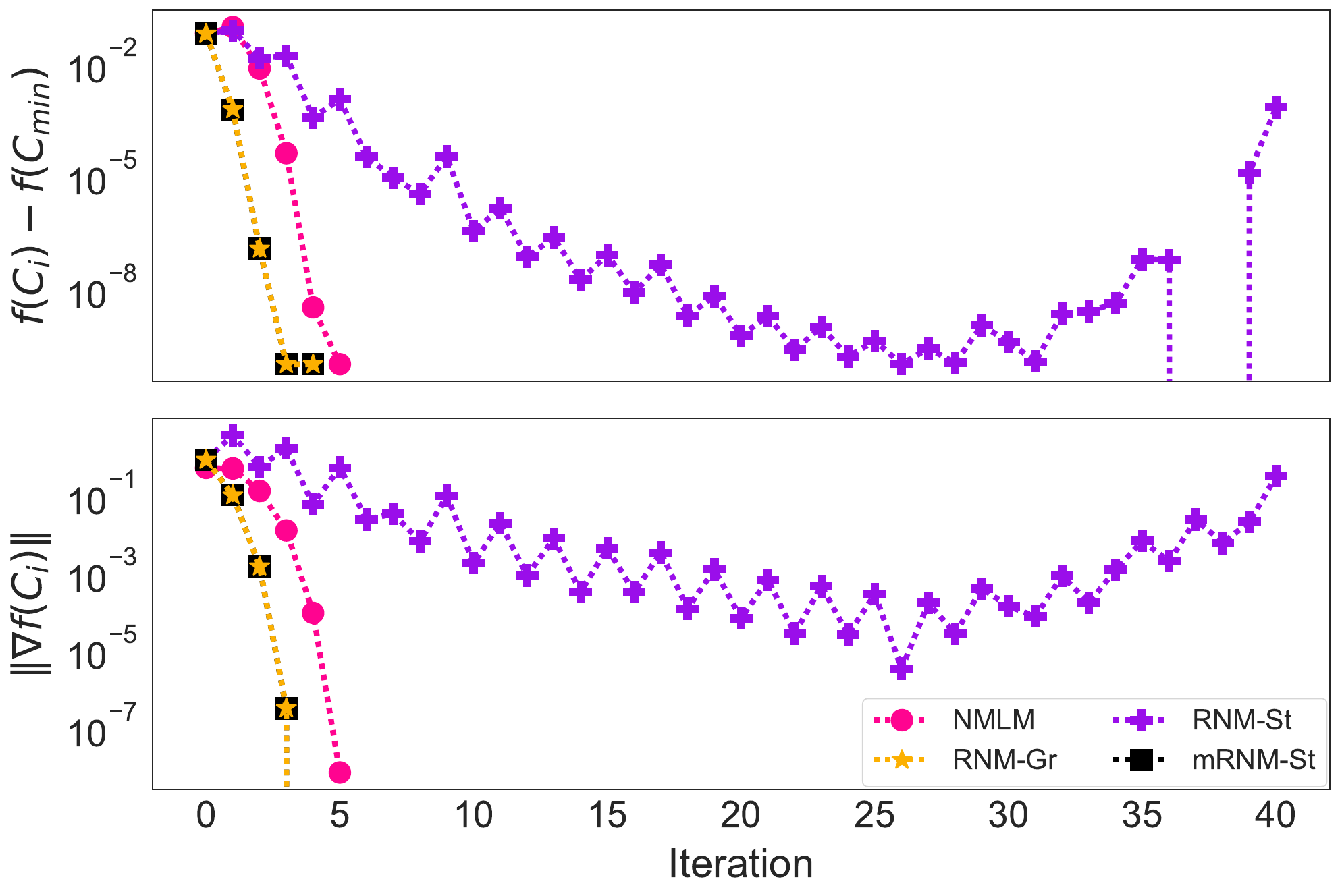}
    \caption{
      Convergence curves for the water molecule.
      Top: energy difference to minimum in each iteration.
      Bottom: gradient norm in each iteration.
    }
    \label{fig:h2o_conv}
\end{figure}

As expected for a Newton's Method, RNM-Gr converges in few iterations.
In contrast, RNM-St exhibits instability:
although it reaches a difference below $10^{-10}$ around the 25th iteration, the difference increases afterwards and the gradient norm never drops below $10^{-6}$.
Moreover, the matrices $C_i$ drift off the Stiefel manifold, measured by the norm $\norm{C_i^{\top}SC_i - \id_N}$ (see Equation~\eqref{eq:Lagr_grad}).
This behavior is expected due to the fact that $\hessSt{f}(\Cmin)$ is not positive-definite.
It is also known that $\dim{\ST} - \dim{\Gr}$ of its eigenvectors are orthogonal to $T_{\Cmin}\gr{N}{d}$ and the remaining $\dim{\Gr}$ eigenvalues and eigenvectors are equal to those of $\hessGr{f}(\Cmin)$ (at least in theory, see~\cite{boumal2023}).
Tables~\ref{tab:water_eigvals} and~\ref{tab:water_eigvecs} confirm this.
The largest eigenvalues and the corresponding eigenvectors of $\hessSt{f}(C_0)$ and $\hessGr{f}(C_0)$ are nearly identical and most of the remaining eigenvectors of $\hessSt{f}(C_0)$ are numerically orthogonal to $T_{\Cmin}\gr{N}{d}$.

Motivated by this, we introduce a modified version of RNM-St in which the eigenvalues below a threshold $\delta = 10^{-8}$ are discarded.
By doing this, we restrict the solution of Newton's Equation~\eqref{eq:newton} to a subspace that is numerically close to $T_C\gr{N}{d}$, as evidenced by the steps shown in Figure~\ref{fig:h2o_conv}.
The results show that mRNM-St mitigates the instability exhibited by RNM-St.

The NMLM initially increases the energy and the deviation from the Stiefel manifold (measured by $\norm{C_i^{\top}SC_i - \id_N}$), but it stabilizes quickly thereafter, requiring only one extra iteration compared to RNM-Gr to reach the solution.

\begin{table}[h]
  \setlength{\tabcolsep}{3.25pt}
  \centering
  \caption{
    Eigenvalues of the Hessians on the Stiefel manifold and on the Grassmannian at the first iteration of RNM-Gr and RNM-St for the water molecule.
  }
  \begin{tabularx}{\textwidth}{l|*{10}{>{\raggedleft\arraybackslash}X}}
    \toprule
    Method & \multicolumn{10}{c}{Eigenvalues} \\
    \midrule
    RNM-Gr & 86.04 & 84.44 & 84.3 & 83.99 & 83.41 & 83.22 & 80.7 & 80.59 & 12.17 & 9.12 \\
    RNM-St & 86.04 & 84.44 & 84.3 & 83.99 & 83.41 & 83.22 & 80.7 & 80.59 & 12.17 & 9.12 \\
    \midrule
    RNM-Gr & 8.95 & 8.31 & 8.11 & 7.9 & 7.81 & 7.72 & 7.09 & 6.85 & 6.37 & 5.84 \\
    RNM-St & 8.95 & 8.32 & 8.11 & 7.9 & 7.81 & 7.72 & 7.09 & 6.85 & 6.38 & 5.85 \\
    \midrule
    RNM-Gr & 5.83 & 5.71 & 5.64 & 5.44 & 5.25 & 5.18 & 5.14 & 5.02 & 4.9 & 4.8  \\
    RNM-St & 5.84 & 5.71 & 5.64 & 5.44 & 5.26 & 5.19 & 5.14 & 5.02 & 4.92 & 4.8 \\
    \midrule
    RNM-Gr & 4.67 & 4.33 & 4.28 & 4.27 & 2.7 & 2.27 & 1.92 & 1.61 & 1.46 & 1.24 \\
    RNM-St & 4.67 & 4.33 & 4.29 & 4.28 & 2.71 & 2.28 & 1.94 & 1.61 & 1.46 & 1.25 \\
    \midrule
    RNM-St & -0.0008 & -0.001 & -0.0023 & -0.0029 & -0.0103 & -0.0188 & -0.021 & -0.0297 & -0.0322 & -0.0531 \\
    \bottomrule
  \end{tabularx}
  \label{tab:water_eigvals}
\end{table}

\begin{table}[h]
  \centering
  \caption{
    Top: inner products between the eigenvectors of the Hessian on the Stiefel manifold and the corresponding eigenvectors on the Grassmannian, following Table~\ref{tab:water_eigvals}.
    Bottom: squared norm of the projection of the remaining eigenvectors on the Stiefel manifold onto the tangent space of the Grassmannian.
  }
  \begin{tabularx}{\textwidth}{*{10}{>{\raggedleft\arraybackslash}X}}
    \toprule
    \multicolumn{10}{c}{Inner product} \\
    \midrule
    1.0000 & 1.0000 & 1.0000 & 1.0000 & 1.0000 & 1.0000 & 1.0000 & 1.0000 & 1.0000 & 0.9999 \\
    \midrule
    0.9998 & 0.9996 & 0.9996 & 0.9995 & 0.9997 & 0.9999 & 0.9999 & 0.9999 & 0.9995 & 0.9135 \\
    \midrule
    0.9137 & 0.9991 & 0.9996 & 0.9994 & 0.9995 & 0.9988 & 0.9998 & 0.9987 & 0.9961 & 0.9978 \\
    \midrule
    0.9995 & 0.9966 & 0.9951 & 0.9967 & 0.9976 & 0.9984 & 0.9955 & 0.9994 & 0.9998 & 0.9952 \\
    \toprule
    \multicolumn{10}{c}{Squared norm of the projection} \\
    \midrule
    0.0001 & 0.0002 & 0.0007 & 0.0008 & 0.0018 & 0.0032 & 0.0035 & 0.0084 & 0.0085 & 0.0210 \\
    \bottomrule
  \end{tabularx}
  \label{tab:water_eigvecs}
\end{table}

%% file: spectrum.tex
\subsection{Spectrum}

The discussion in Subsection~\ref{sec:res_water} on the water molecule shows that RNM-Gr and mRNM-St are very similar because the spectra of the Hessians on the Grassmannian and the Stiefel manifold, excluding the smallest $\dim{\ST} - \dim{\Gr}$ eigenvalues and their corresponding eigenvectors, are numerically close.
Since examining the individual spectrum of each molecule in the dataset is impractical, we extend this analysis to the full dataset using the root mean squared difference between the eigenvalues of both Hessians:
\begin{equation} \label{eq:D_spectrum}
    D \coloneqq \sqrt{\frac{1}{\dim{\Gr}} \sum_{i=1}^{\dim{\Gr}} \left(\lambda_i^{\Gr} - \lambda_i^{\ST}\right)^2}\,.
\end{equation}
We assume the eigenvalues $\{\lambda_i^{\Gr}\}_{i=1}^{\dim{\Gr}}$ and $\{\lambda_i^{\ST}\}_{i=1}^{\dim{\ST}}$ are in descending order, as in Table~\ref{tab:water_eigvals}, and we use only the first $\dim{\Gr}$ eigenvalues of $\hessSt{f}(C)$ to compute $D$.

Figure~\ref{fig:D_per_mol} presents the values of $D$ for all molecules in the dataset.
All eigenvalues were computed at the initial guess $C_0$.
The results show that $D$ is small for the entire dataset:
no value exceeds $0.08$ and 75\% fall below $0.02$, indicating that the spectra of both Hessians are numerically close.
In addition to plotting the values of $D$, we use colors and markers to convey more information about how the two methods compare for each molecule.
In general, molecules with low values of $D$ exhibit very similar convergence patterns for RNM-Gr and mRNM-St.
Both methods tend to converge to the same critical point and very often in the same number of iterations.
In some cases, however, RNM-Gr requires one extra iteration (highlighted in yellow).
Conversely, calculations where one method fails to converge or converges to a different critical point tend to exhibit larger values of $D$.
This is consistent with the fact that a large $D$ indicates the solutions to Equation~\eqref{eq:solve_newton_eigenvalues} on the two manifolds are not the same, leading to different optimization steps.

Some outliers are worth mentioning.
For instance, the H$_2$CO molecule has a small $D \approx 0.01$, yet mRNM-St and RNM-Gr converge to different critical points.
In contrast, the NH molecule exhibits the largest $D$ in the dataset, but both methods converge to the same critical point, with RNM-Gr requiring two extra iterations.
While a detailed analysis of each case is beyond the scope of this work, the results suggest that a closer inspection of the Hessians spectra may help clarify the behavior of the outliers.

\begin{figure}[p]
  \caption{
    The root mean squared difference, $D$, between the eigenvalues of the Hessians on the Stiefel manifold and the Grassmannian.
  }
  \includegraphics[width=1.0\textwidth]{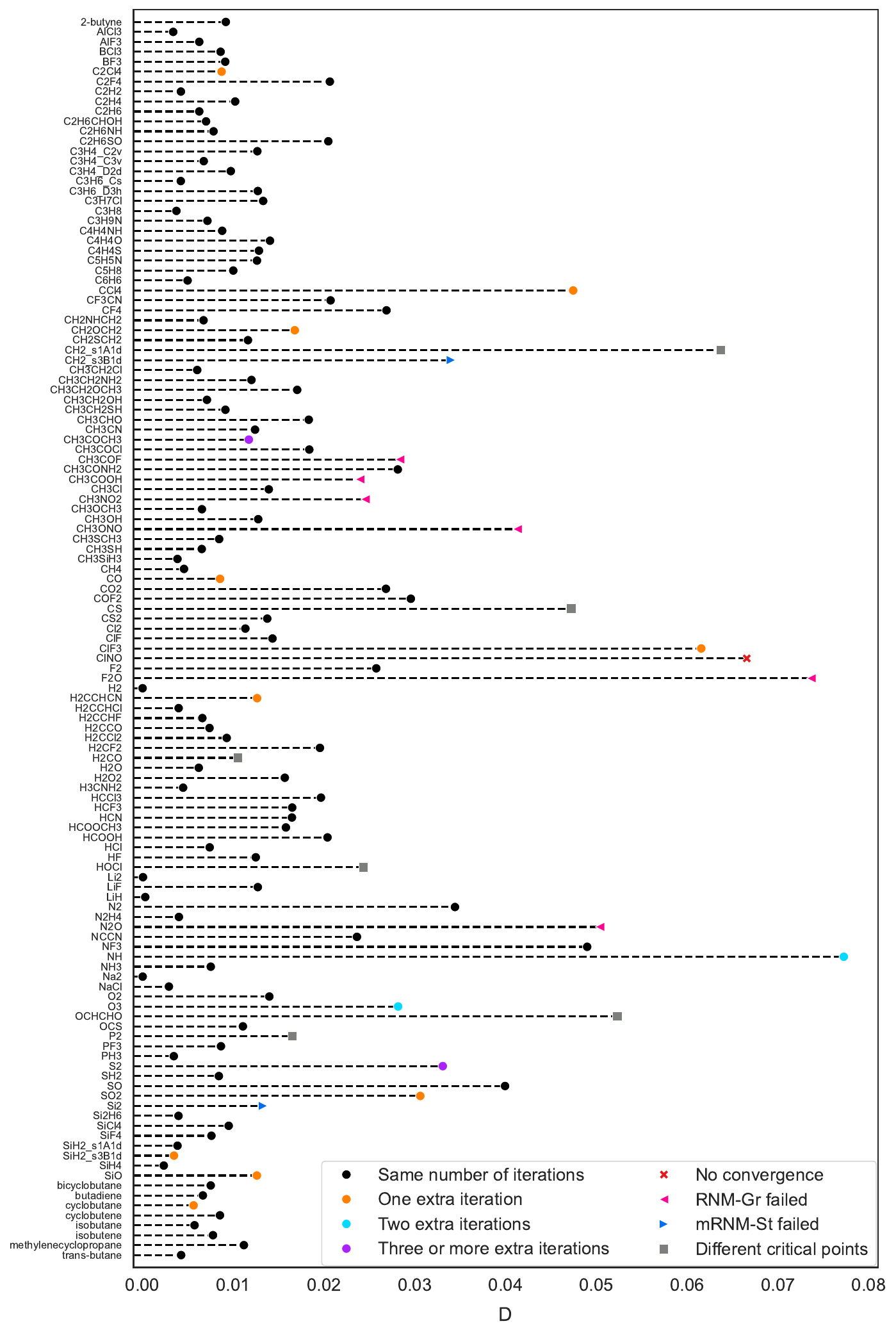}
  \label{fig:D_per_mol}
\end{figure}

%% file: res_statistics.tex
\subsection{Statistical analysis}

We now turn to the overall statistics for the results obtained from all 125 molecules in the G2/97 dataset.
Table~\ref{tab:res_statistics} shows the number of converged calculations and the average number of iterations.
Recall that, since we do not know the ground truth minimum, we define convergence as the gradient norm dropping below $10^{-8}$.
Note also that the average number of iterations is computed using only those molecules for which the method converged.
Moreover, all calculations used the initial guess obtained from the SAD procedure, as previously mentioned.

The results in Table~\ref{tab:res_statistics} show that the Riemannian methods perform similarly, which is consistent with the results for the water molecule.
Both methods converge for more than 95\% of the molecules, with mRNM-St converging in few more cases than RNM-Gr.
In contrast, Newton's Method with Lagrange multipliers (NMLM) converged for only 73.6\% of the molecules.
The Riemannian methods also require, on average, fewer iterations to converge.
Since each iteration involves inverting the Hessian, reducing the number of iterations by nearly half represents a meaningful computational advantage.
Overall, we see that the Riemannian methods clearly outperform the NMLM method.

\begin{table}[h]
    \centering
    \caption{
        Number of converged molecules and average number of iterations for each method.
        Averages are computed over converged cases only.
    }
    \begin{tabularx}{\textwidth}{l >{\raggedleft\arraybackslash}X >{\raggedleft\arraybackslash}X}
        \toprule
        Method & Number of converged cases & Average number of iterations \\
        \midrule
        NMLM    & 92/125  (73.6\%) & 7.576 \\
        RNM-Gr  & 118/125 (94.4\%) & 4.220 \\
        mRNM-St & 122/125 (97.6\%) & 4.426 \\
        \bottomrule
    \end{tabularx}
    \label{tab:res_statistics}
\end{table}

%% file: ablation.tex
\subsection{Ablating \texorpdfstring{$\delta$}{delta} in mRNM-St}

The Riemannian Newton's Method on the Stiefel manifold (RNM-St) never converges.
This motivated us to modify the method by truncating the spectrum of the Hessian.
Specifically, we discard eigenvalues smaller than a threshold $\delta$ when solving Newton's Equation~\eqref{eq:newton} (see Section~\ref{sec:hessian} for details).
By default, we use $\delta = 10^{-8}$ since it is numerically close to zero and is sufficient to ensure the Hessian is positive-definite.
However, we noticed that increasing this value may lead some molecules that previously did not converge to converge.
We therefore tested the impact of changing $\delta$ in the performance of mRNM-St over the entire dataset.

Table~\ref{tab:ablation_stiefel} shows that increasing $\delta$ actually reduces the number of converged cases.
However, the average number of iterations \emph{among converged cases} decreases as $\delta$ increases.
This happens because, for small values of $\delta$, some molecules require many iterations to converge, which increases the average.
These are exactly the molecules that fail to converge for larger values of $\delta$, and thus the cases that still converge at those values, do so in fewer iterations.
This indicates that the more difficult cases (those requiring many iterations) depend on the small eigenvalues.
If those are excluded by the threshold, convergence fails.
Only in a few cases (2 molecules) increasing $\delta$ results in convergence where it previously failed, but even in those cases the optimization failed after further increasing $\delta$.
These results suggest that using a small $\delta$ is a reasonable starting point and, if convergence is not achieved, then increasing it may help in some cases.
Moreover, the results did not change for $\delta$ between $10^{-8}$ and $10^{-3}$.

Detailed individual results for each molecule, whose statistical analysis have been discussed in this and in the previous subsections, are provided in the Supporting Information.

\begin{table}[h]
  \caption{Number of converged molecules for mRNM-St$_\delta$ and the average number of iterations for several values of $\delta$.}
  \centering
  \begin{tabularx}{\textwidth}{r >{\raggedleft\arraybackslash}X >{\raggedleft\arraybackslash}X}
      \toprule
      \texorpdfstring{$\delta$}{Delta} &  Number of converged cases & Average number of iterations \\
      \midrule
      $0.001$ & 122/125  (97.6 \%) &  4.43  \\
      $0.1$   & 119/125  (95.2 \%) &  4.25  \\
      $0.2$   & 117/125  (93.6 \%) &  4.17  \\
      $0.3$   & 114/125  (91.2 \%) &  4.16  \\
      $0.4$   & 114/125  (91.2 \%) &  4.13  \\
      $0.5$   & 107/125  (85.6 \%) &  4.02  \\
      $0.6$   & 104/125  (83.2 \%) &  3.97  \\
      $0.7$   & 100/125  (80.0 \%) &  3.99  \\
      $0.8$   & 86/125   (68.8 \%) &  3.95  \\
      $0.9$   & 80/125   (64.0 \%) &  3.93  \\
      $1.0$   & 68/125   (54.4 \%) &  3.84  \\
      \bottomrule                           
  \end{tabularx}
  \label{tab:ablation_stiefel}
\end{table}

%% file: profiling.tex
\subsection{Performance profile for the number of iterations}

Although the previous subsections already indicate that the Riemannian methods outperform the classical Newton's Method with Lagrange multipliers, we also carried a \emph{performance profile} analysis, which is a standard practice in the literature (see~\cite{dolan2002} for details).

Recall that performance profiles compute a cumulative distribution of best results.
For a given $\tau \geq 1$, $\rho(\tau)$ is defined as the proportion of cases in which a method's performance is better than $\tau$ times the best case.
Naturally, the notion of “best” depends on the metric under consideration.
In this subsection, we use the number of iterations required for convergence, but other metrics will be considered later.

Formally, let $p$ index the 125 molecules in the dataset.
For each method $s$, we define:
\begin{equation} \label{eq:perf_prof_rho}
    \rho_s(\tau)
    \ceq \frac{\text{number of molecules $p$ with $r_{p,s} \leq \tau$}}{\text{number of molecules in the dataset}},
\end{equation}
being $r_{p,s}$ the performance ratio, defined as:
\begin{equation} \label{eq:perf_prof_r}
  r_{p,s} \ceq \frac{\text{number of iterations for molecule $p$ and method $s$}}
  {\text{minimum number of iterations among all $s$ for $p$ fixed}}.
\end{equation}

For example, for $s =$ RNM-Gr, the molecule H$_2$ converges in 2 iterations, which is the best case.
There are 20 molecules that converge in 3 iterations, 71 in 4 iterations, and so on.
In all these cases RNM-Gr is the best method (or it ties with mRNM-St).
Therefore:\footnote{We omit the remaining terms for brevity.}
\begin{equation}
    \rho_{\text{RNM-Gr}}(\tau) =
    \begin{cases}
        \frac{1}{125}       = 0.008, & \text{if $\tau \in [1, \frac{3}{2})$}, \\
        \frac{1+20}{125}    = 0.168, & \text{if $\tau \in [\frac{3}{2}, \frac{4}{2})$}, \\
        \frac{1+20+71}{125} = 0.736, & \text{if $\tau \in [\frac{4}{2}, \frac{5}{2})$}.
    \end{cases}
\end{equation}
Observe that, the faster $\rho_s(\tau)$ increases, the better the method, since more cases converge in fewer iterations.
Furthermore, as we increase $\tau$, $\rho_s(\tau)$ eventually reaches the proportion of molecules for which the method $s$ converged.

Figure~\ref{fig:perf_prof_niter} shows that both RNM-Gr and mRNM-St exhibit similar performance, with $\rho_{\text{mRNM-St}}$ slightly higher than $\rho_{\text{RNM-Gr}}$ for large $\tau$.
We also observe that increasing $\delta$ reduces the overall performance of mRNM-St for large $\tau$.
Finally, we can see that once again NMLM is outperformed by the Riemannian methods:
its $\rho$ reaches a lower value for large $\tau$, as expected, but its curve also does not increase as fast as RNM-Gr and mRNM-St.

\begin{figure}[h]
    \caption{
        Performance profile for the number of iterations required for convergence.
        Each curve shows the proportion $\rho$ of molecules in which a method's performance is better than $\tau$ times the best case.
    }
    \includegraphics[width=1.0\textwidth]{"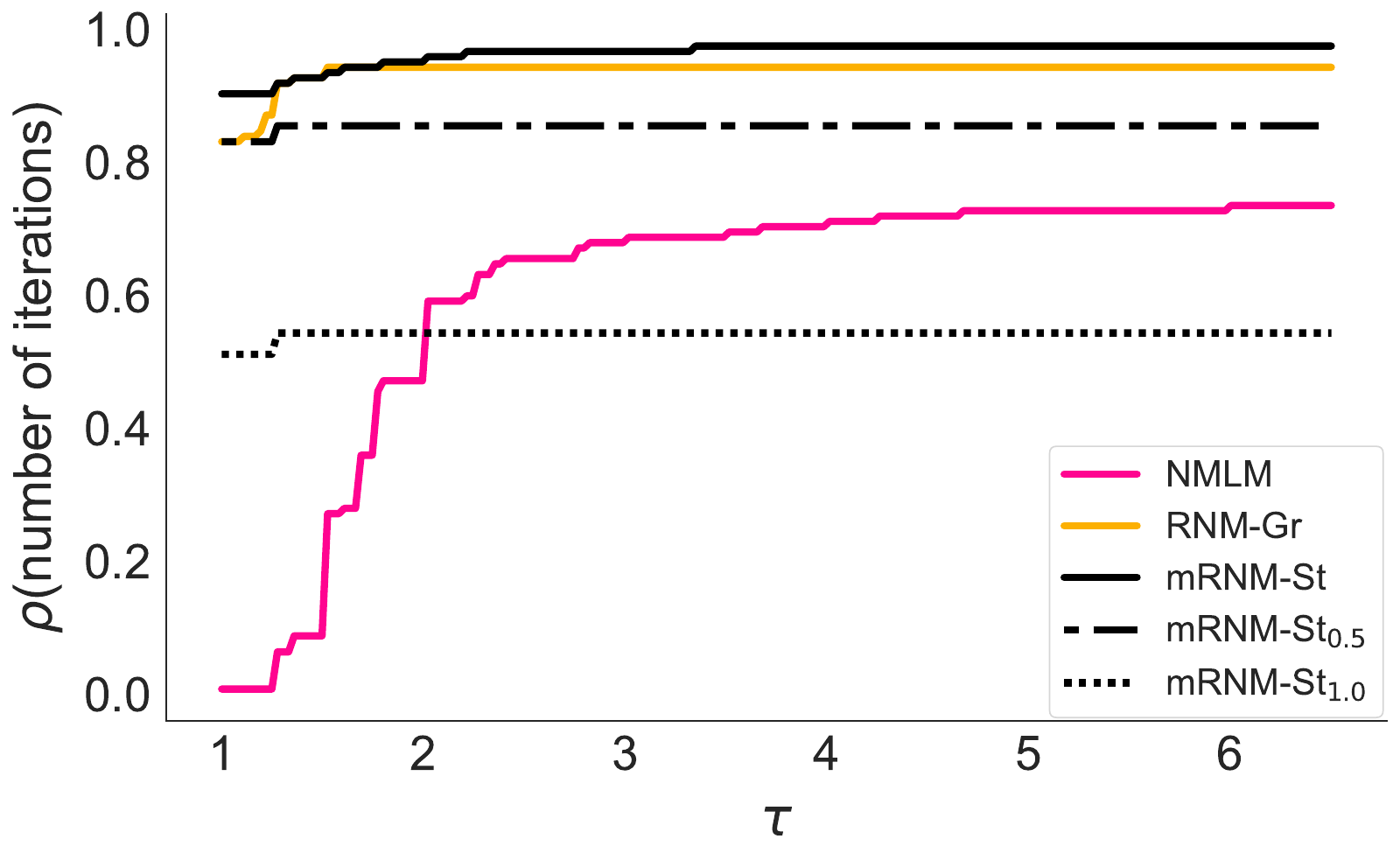"}
    \label{fig:perf_prof_niter}
\end{figure}

%% file: res_conv_viz.tex
\subsection{Characterizing the convergence neighborhood} \label{sec:conv_neigh_water}

The statistical results and the performance profile discussed in the previous sections show a clear advantage for the Riemannian methods.
However, these results strongly depend on the choice of initial guess.
To decouple convergence performance from the initial guess, we will make use of the fact that there exists a convergence neighborhood around a critical point such that, if the starting point lies within this region, Newton's Method is guaranteed to converge to that critical point (see \cite{galantai2000, tanabe1985}).
Explicitly describing such neighborhoods is very difficult, especially when the domain of the cost function is a high-dimensional Riemannian manifold.
Therefore, we take a numerical approach to estimate and compare the convergence neighborhoods of the methods under study.
This allows us to determine which method has the largest convergence region, which in turn suggests greater robustness, since convergence does not necessarily require starting close to the solution.

We now describe the approach in a general form, as it may be applied in other settings.
Let $(X, \inner{\cdot}{\cdot})$ be a complete Riemannian manifold, let $f: X \to \br$ be a smooth cost function, and let $x_*$ be a critical point of $f$, meaning $\riemgrad{f}(x_*) = 0$.
Consider an iterative algorithm $A$ that, given $x_0 \in X$, produces a sequence $\{x_k \in X : k \in \bn\}$ which \emph{may} converge to $x_*$.
For a fixed small $\eps > 0$, we define the number of iterations for $A$ to converge to $x_*$ as the smallest $n \in \bn$ such that $\norm{\riemgrad{f}(x_n)} < \eps$.
Note that this is consistent with the convergence criterion used previously, where $\eps = 10^{-8}$.
Taking all of this into account, we aim to describe the following set:
\begin{equation}
    V_{x_*} \ceq \left\{x \in X :
    \begin{array}{c}
        \text{$x = \exp_{x_*}(v)$ for some $v \in T_{x_*}X$ and $A$ converges to $x_*$} \\
        \text{if it starts at $\exp_{x_*}(tv)$ for $t \in [0,1]$}
    \end{array}
    \right\}.
\end{equation}

This set $V_{x_*}$ can be interpreted as the largest star-shaped set centered at $x_*$ such that, if the algorithm $A$ starts at any $x \in V_{x_*}$, it converges to $x_*$.
While the true convergence region may not be star-shaped, in this work we will focus on this simplified region.
We also examine how the number of iterations $n$ varies as a function of $x$.
To estimate this region numerically, we consider an orthonormal basis $\{u_i\}_{i=1}^n$ for $T_{x_*}X$, and define $\{v_j\} \ceq \{u_i\}_{i=1}^n \cup \{-u_i\}_{i=1}^n$ as the set of directions we will move along to obtain $V_{x_*}$.
Once we have $v_j$, we consider several values of $t > 0$ and calculate $x = \exp_{x_*}(tv_j)$ to use as the initial guess of algorithm $A$.

Figure~\ref{fig:h2o_conv_neigh} depicts the convergence neighborhood for the water molecule.\footnote{Similar plots for all molecules in the dataset are provided in the Supporting Information.}
Each bar corresponds to one basis vector and the optimization was performed for increasing values of $t$, stepping by $0.05$ in both positive and negative directions from the critical point at $t = 0$.
The critical points were obtained using RNM-Gr and, for the cases in which this method did not converge, we used mRNM-St.
There was one molecule, ClNO, that none of the methods converged and we excluded it from the following analysis.
Note that the plot is not symmetric, i.e., the convergence radius can differ for positive and negative $t$.
The color encodes the number of iterations needed to achieve convergence.
To better understand the plot, let us focus on the left-most bar in the RNM-Gr panel.
This bar corresponds to the geodesic emanating from the critical point in the direction of the first basis vector.\footnote{We fix an order for the basis as indicated in Session~\ref{sec:hessian}.}
RNM-Gr converges to the critical point for all starting points along this geodesic with $-0.25 < t < 0.45$.
Within the interval $|t| \leq 0.1$, convergence occurs in 3 iterations.
For $|t| < 0.25$, it takes 4 iterations, and so on.
Each vertical bar corresponds to a different direction, and the interpretation is similar.
An interesting behavior is that the method converges in 5 iterations for $t = 0.3$, but in 4 iterations for both $t = 0.35$ and $t = 0.4$.
We do not have an explanation for this phenomenon and it occurs for all methods across the dataset.

Figure~\ref{fig:h2o_conv_neigh} also shows that RNM-Gr and mRNM-St have similar, but not identical, convergence neighborhoods, while NMLM's neighborhood is remarkably different.
Comparing the two Riemannian methods, we see that mRNM-St has a larger convergence neighborhood.
While both methods perform similarly for $|t| < 0.3$, RNM-Gr starts to fail at $t \approx 0.4$, whereas mRNM-St continues to converge until about $|t| \approx 0.6$.
For mRNM-St, there is a wide region where the method requires more than 5 iterations to converge, whereas RNM-Gr typically fails once more than 5 iterations are needed (at least for the water molecule).
NMLM often requires 4 or more iterations even for small values of $|t| = 0.1$, and it requires 6 or more iterations already at $|t| \approx 0.3$.
As will be shown next, the average value of $t$ (across directions) for which NMLM converges is slightly smaller than for mRNM-St.
Nevertheless, there are directions along which NMLM converges for surprisingly large values of $t$.

To quantify the size of each convergence neighborhood, we consider the following metrics.
For each direction vector $v_i$, let $R_i$ be the largest $t$ for which the algorithm converges to the critical point.
We define the \emph{minimum convergence radius} $R_{\mathrm{min}}$, the \emph{average convergence radius} $R_{\mathrm{avg}}$, and the \emph{maximum convergence radius} $R_{\mathrm{max}}$ as the minimum, average, and maximum of the set $\{R_i\}$, respectively.
It is worth pointing out that $R_{\mathrm{max}}$ is a best-case scenario metric and that $R_{\mathrm{min}}$ and $R_{\mathrm{avg}}$ reflect better the actual convergence neighborhood.
We will report all values for the sake of completeness, though.

Table~\ref{tab:R_conv_H2O} presents the metrics defined above for the water molecule.
In this particular case, RNM-Gr exhibits the worst performance, with the smallest values in all metrics.
On the other hand, mRNM-St and NMLM are on par:
mRNM-St has a slightly higher $R_{\mathrm{min}}$ while NMLM has a way larger $R_{\mathrm{max}}$.
We can see that NMLM in this case is competitive with the Riemannian methods.
However, as shown in the next subsection, its performance is less robust when evaluated across the entire dataset.

\begin{figure}
    \caption{
        Convergence neighborhoods of RNM-Gr, mRNM-St, and NMLM for the water molecule.
        Vertical bars correspond to directions in the tangent space at the critical point, rows represent increasing step sizes along that direction.
        Colors indicate the number of iterations required for convergence.
    }
    \includegraphics[width=\linewidth]{"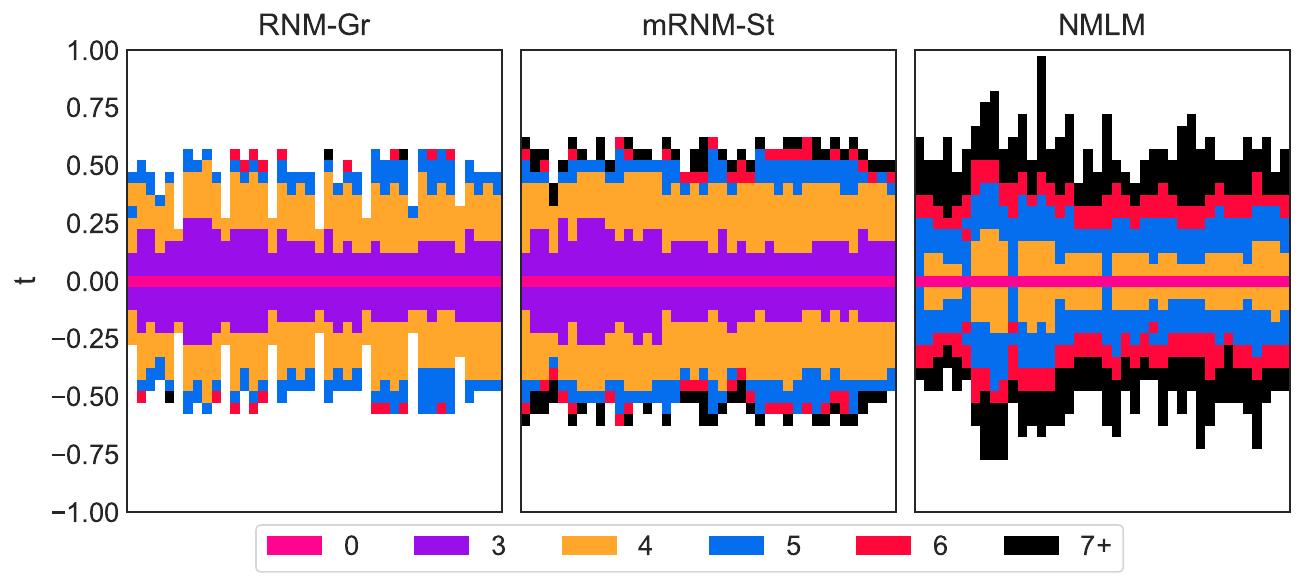"}
    \label{fig:h2o_conv_neigh}
\end{figure}

\begin{table}[h]
    \caption{
        Minimum ($R_{\mathrm{min}}$), average ($R_{\mathrm{avg}}$), and maximum ($R_{\mathrm{max}}$) convergence radii across directions for the water molecule.
    }
    \centering
    \begin{tabularx}{\textwidth}{l >{\raggedleft\arraybackslash}X >{\raggedleft\arraybackslash}X >{\raggedleft\arraybackslash}X}
        \toprule
        Method & $R_{\mathrm{min}}$ & $R_{\mathrm{avg}}$ & $R_{\mathrm{max}}$ \\
        \midrule
        NMLM    & 0.35 & 0.55  & 0.95 \\
        RNM-Gr  & 0.2  & 0.449 & 0.55 \\
        mRNM-St & 0.4  & 0.55  & 0.6  \\
        \bottomrule
    \end{tabularx}
    \label{tab:R_conv_H2O}
\end{table}

%% file: profiling_r.tex
\subsection{Analyzing the convergence radii}

In this subsection, we analyze the convergence radii across the entire dataset.
Figure~\ref{fig:all_conv_radii} shows the minimum, average, and maximum convergence radii for each molecule individually, while Figure~\ref{fig:boxplot_all_conv_radius} is a boxplot with the distribution of these radii across the entire dataset.

For the average convergence radius,  $R_{\mathrm{avg}}$, the median of mRNM-St is $0.54$, which is approximately $0.1$ higher than those of RNM-Gr ($0.46$) and NMLM ($0.43$).
That said, although the medians of RNM-Gr and NMLM are close, NMLM exhibits higher variance, indicating that its performance is substantially more sensitive to the specific problem at hand.

When considering the minimum convergence radius, NMLM has a median $R_{\mathrm{min}}$ of $0$, revealing that, in most cases, there exists at least one direction in which the method fails to converge.
In contrast, RNM-Gr and mRNM-St have medians of $0.15$ and $0.05$, respectively.
Both Riemannian methods thus outperform NMLM, demonstrating that, even for unfavorable initial directions, convergence can be achieved as long as the starting point is sufficiently close to the critical point.
It is worth noting that $R_{\mathrm{min}}$ reflects a worst-case scenario.
Moreover, since Newton's Method is not a globally convergent algorithm and is typically used in the final iterations of an optimization procedure, these results suggest that Riemannian methods are more reliable for guiding the last iterations to convergence.

For the maximum convergence radius, $R_{\mathrm{max}}$, NMLM stands out with a median of $0.8$, surpassing RNM-Gr ($0.6$) and mRNM-St ($0.65$).
This indicates that, for each molecule, there is at least one direction in which NMLM converges for relatively large values of $t$, suggesting that NMLM outperform the Riemannian methods in the best-case scenario.

We also computed performance profiles based on the convergence radii.
Before discussing these results, though, it is important to clarify that, unlike traditional performance profiles~\cite{dolan2002}, where lower values indicate better performance, in this case larger convergence radii are preferable.
Therefore, we invert the following ratio:
\begin{equation} \label{eq:perf_prof_r_inv}
  r_{p,s} = \frac{\text{maximum convergence radius among all $s$ for $p$ fixed}}
{\text{convergence radius for molecule $p$ and method $s$}}.
\end{equation}
With this definition, $r_{p,s} = 1$ corresponds to the best case and higher values indicate worse performance.
For instance, $r_{p,s} = 2$ means that the convergence radius for molecule $p$ using method $s$ is half that of the best-performing method for the same molecule.

Figure~\ref{fig:perfprof_all_conv_radius} shows the performance profiles for the minimum, average, and maximum convergence radii across the dataset.
The results once again confirm that, for both $R_{\mathrm{min}}$ and $R_{\mathrm{avg}}$, NMLM is outperformed by the Riemannian methods.
We can see that approximately $80\%$ of the molecules have at least one direction in which NMLM does not converge.
In summary, NMLM is less robust than the Riemannian methods, as evidenced by its higher sensitivity to the specific problem (reflected in the high variance of $R_{\mathrm{avg}}$) and to the choice of initial guess.

\begin{figure}[h!]
  \caption{
    Minimum, average, and maximum convergence radii for each molecule and method.
  }
  \includegraphics[width=1.0\textwidth]{"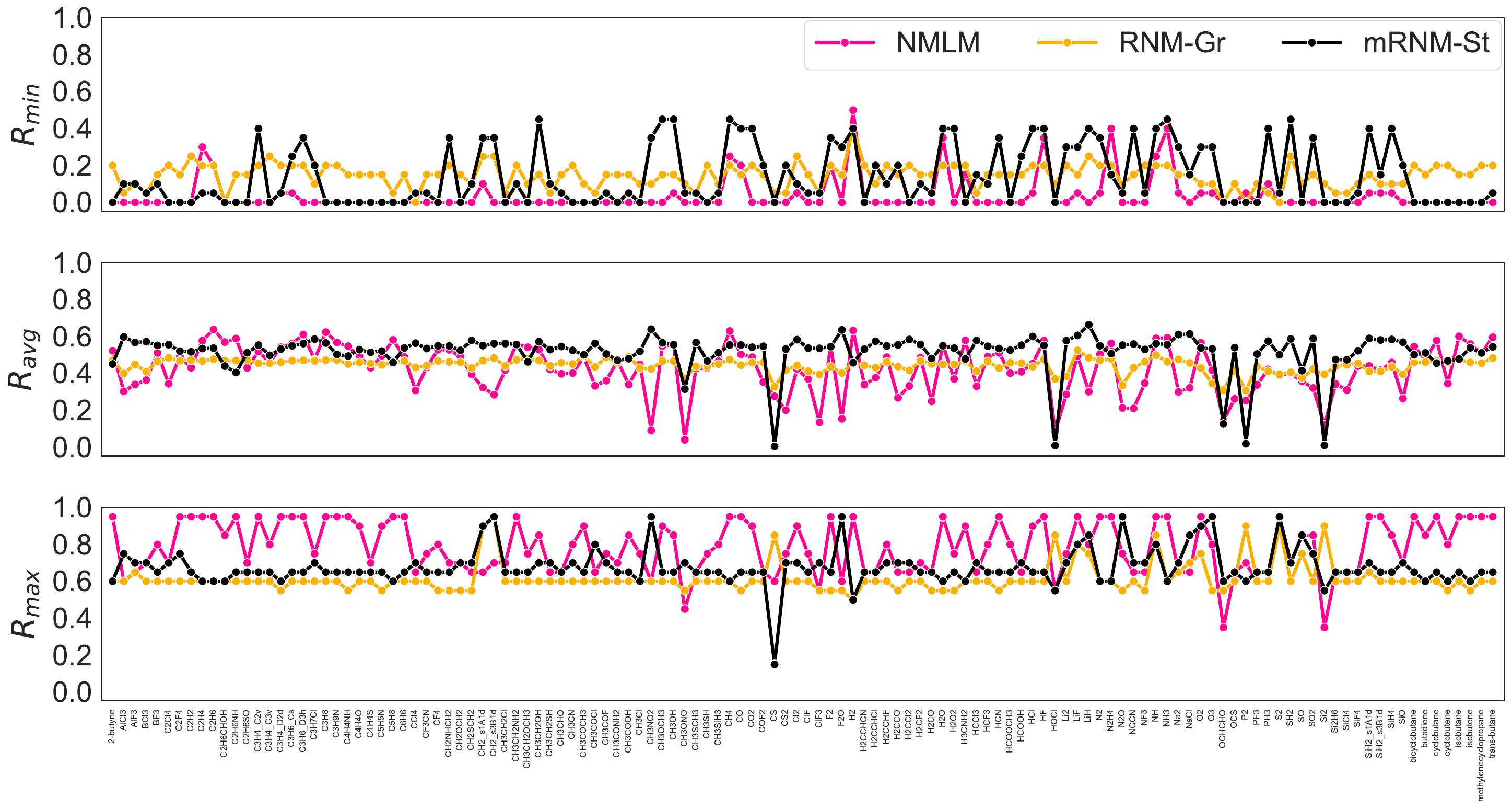"}
  \label{fig:all_conv_radii}
\end{figure}

\begin{figure}[h!]
  \caption{
    Boxplot showing the distribution of minimum, average, and maximum convergence radii ($R$) across all molecules in the dataset for each method (NMLM, RNM-Gr, and mRNM-St).
    The central line in each box indicates the median, the box represents the interquartile range, and the whiskers show the data range excluding outliers.
  }
  \includegraphics[width=1.0\textwidth]{"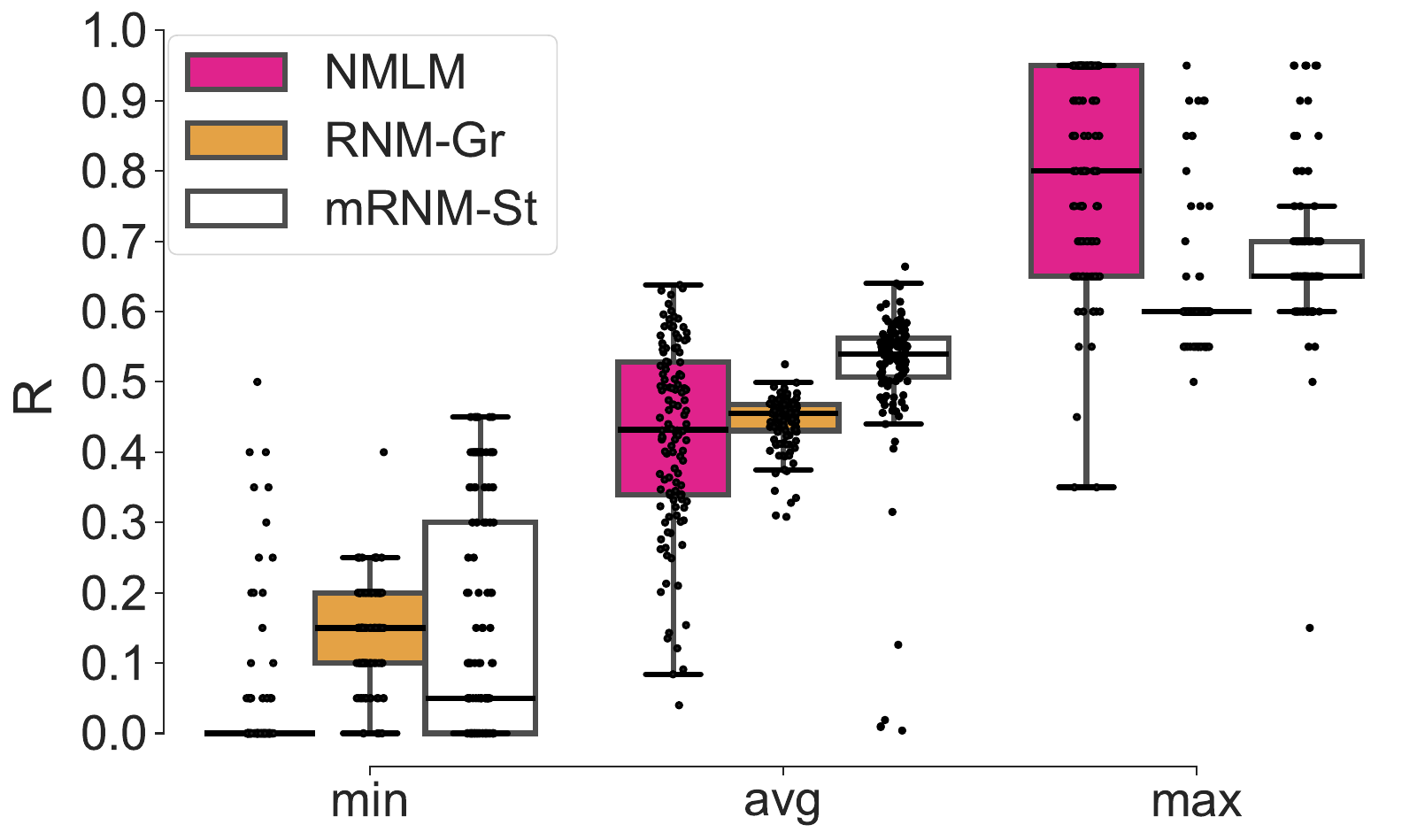"}
  \label{fig:boxplot_all_conv_radius}
\end{figure}

\begin{figure}[h!]
  \caption{
    Performance profiles comparing the minimum, average, and maximum convergence radii of NMLM, RNM-Gr, and mRNM-St.
    The $y$-axis is shared across columns, while the $x$-axes differ.
  }
  \includegraphics[width=1.0\textwidth]{"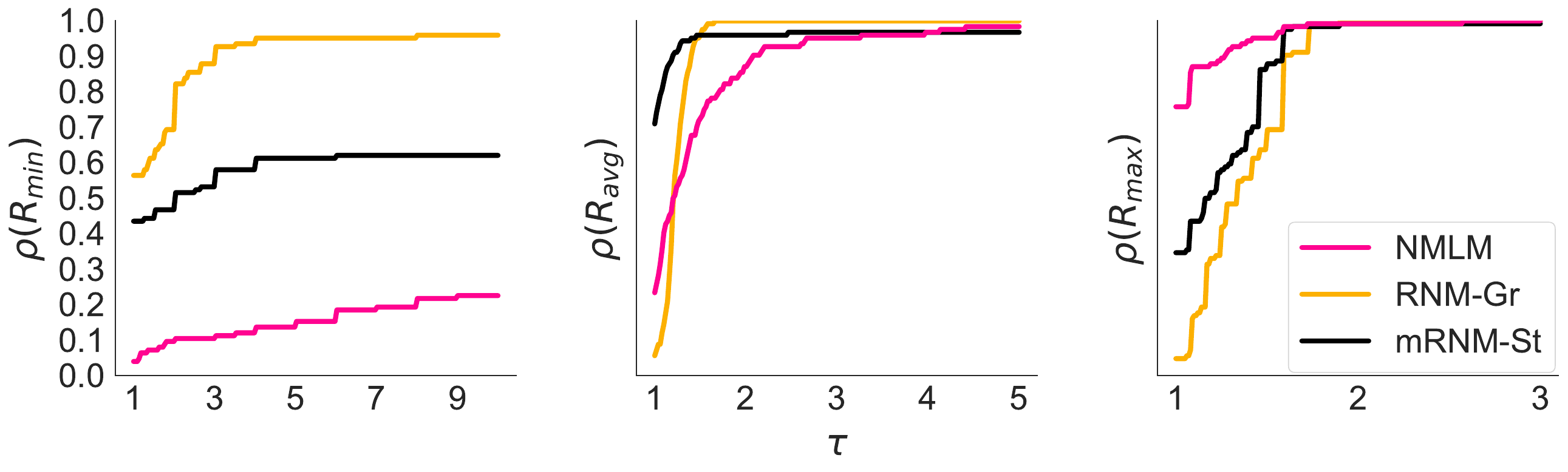"}
  \label{fig:perfprof_all_conv_radius}
\end{figure}

%% file: conclusions.tex
\section{Conclusions and final remarks}

This work presented a thorough numerical study of the Riemannian Newton's Method, with two main goals:
to compare its performance against the method of Lagrange multipliers, and to better understand the convergence issues that arise when the objective function is defined on a quotient manifold.

We carried out the experiments on the Hartree–Fock optimization problem from Quantum Chemistry, which can be formulated on both the Stiefel manifold and the Grassmannian.
Computations were performed on 125 molecules from the G2/97 dataset.
To evaluate the performance of the methods, we considered two metrics:
the percentage of cases that successfully converged and the number of iterations required to converge.
All methods started from the same initial guess.
We also examined the convergence behavior by numerically estimating the convergence radii of each method.
The results show that the Riemannian methods outperform the Lagrange multipliers method across all evaluation criteria.

Although it is well known that lifting a cost function from a quotient manifold to the total space can lead to numerical issues (since the function is invariant under equivalence classes), we showed that these issues can be mitigated by solving Newton’s Equation in the subspace spanned by the Hessian’s eigenvectors with eigenvalues above a small threshold.
This is relevant because it is often easier to formulate and implement Newton’s Method on a total space than on a quotient manifold.
Moreover, the modified Newton's Method on the Stiefel manifold not only fixes the convergence problem, but it also slightly outperforms Newton’s Method on the Grassmannian.